\documentclass[noinfoline]{article}
\RequirePackage[OT1]{fontenc}
\RequirePackage{amsthm,amsmath,natbib}
\setcitestyle{square,numbers}
\RequirePackage[colorlinks,citecolor=blue,urlcolor=blue]{hyperref}
\usepackage{enumerate}
\usepackage{amsmath}
\usepackage{amsfonts}
\usepackage{geometry}
\newtheorem{theorem}{Theorem}
\newtheorem{lemma}[theorem]{Lemma}
\newtheorem{proposition}[theorem]{Proposition}
\newtheorem{corollary}[theorem]{Corollary}
\theoremstyle{definition}
\newtheorem{definition}[theorem]{Definition}

\newtheorem{hypothesis}{Hypothesis}
\theoremstyle{remark}
\newtheorem{remark}[theorem]{Remark}
\numberwithin{equation}{section} 

\newcommand{\Real}{\mathbb{R}}

\newcommand{\R}{\mathbb{R}}

\begin{document}
	\title{Integration by parts formulas and Lie's symmetries of SDEs}
	\author{Francesco C. De Vecchi \thanks{Dipartimento di Matematica ``Felice Casorati'', Universit\`a degli Studi di Pavia
  \textit{francescocarlo.devecchi@unipv.it}}\and  Paola Morando \thanks{DISAA, Universit\`a degli Studi di Milano, \textit{paola.morando@unimi.it}} \and Stefania Ugolini \thanks{Dipartimento di Matematica, Universit\`a degli Studi di Milano, \textit{stefania.ugolini@unimi.it}} }
\date{}

\maketitle
\begin{abstract}
A strong quasi-invariance principle and a finite-dimensional integration by parts formula as in the Bismut approach to Malliavin calculus are obtained through a suitable application of Lie's symmetry theory to autonomous stochastic differential equations.  The main stochastic, geometrical and analytical aspects of the theory are discussed and  applications to some Brownian motion driven stochastic models are provided.
\end{abstract}

\section{Introduction}

 Born in 1970s Malliavin, calculus very soon became an (infinite dimensional) analysis on Wiener space (see, e.g.,  \cite{cruzeiro2006malliavin,DiNunno,Norris,malliavin1978stochastic,nualart2006malliavin,watanabe1984lectures}).
 The main tool of this calculus  is certainly the integration by parts formula, and while important applications are the study of the regularity of the probability density of random variables and the related computation of conditional expectations, nowadays many other applications to Stochastic Partial Differential Equations (SPDEs) and to probabilistic numerical methods are available (see, e.g., \cite{malliavin2006stochastic,sanz2005malliavin}).
 Malliavin calculus has a deep connection with Wiener chaos decomposition and can be introduced starting from it (see \cite{DiNunno}).
 It was also applied to diffusion processes which are solutions to Stochastic Differential Equations (SDEs), obtaining the conditions (Hormander's condition) under which the density of the law of the process is smooth and satisfies exponential bounds together with its derivatives, and thus providing the famous probabilistic proof of Hormander theorem (see, e.g., \cite{DiNunno, Watanabe}).
Quasi-invariance properties of diffusion processes with respect to flows generated by vector fields has been established in an abstract Wiener space by \cite{cruzeiro2006malliavin} and in the classical path-space, e.g., by \cite{bell2012malliavin, driver1994cameron}, and \cite{hsu1995quasi}. Another fundamental result in Malliavin calculus is the integration by parts formula for functional of Brownian motion, through which it is possible to prove the closability of Malliavin derivatives and the well-definition of the Ornstein Uhlenbeck operator (see, e.g., \cite{Bogachev2010,nualart2006malliavin}). For this reason, it is interesting to obtain (more or less explicit) integration by parts formulas involving different stochastic processes or probability measures in infinite dimensional spaces (see, e.g., \cite{Bogachev2010} for the problem of integration by parts formula for generic measure on infinite dimensional spaces, \cite{Altman2020,Gordina2023,Zambotti2005} for examples of integration by parts formulas of stochastic processes, \cite{AltmanZambotti2020,Zambotti2003} for the integration of Bessel process applied to the study of SPDEs, and \cite{Bogachev2010,de2022singular,GubinelliHofmanova2021} for applications of integration by parts formula to quantum field theory). \\

 In the Bismut variational approach to this calculus (see \cite{bismut1981martingales}), the integration by parts formula is derived from a fundamental (strong) quasi-invariance principle, which is based on the well-known invariance property of Wiener law under a  measure change via Girsanov theorem. Indeed, the Brownian motion, the filtration generated by it and the functionals of Brownian motion are at the core of Malliavin calculus. See \cite{leandre2019bismut} and references therein for a recent review on Bismut's way to Malliavin calculus.\\

In spite of the fact that we are dealing with an infinite-dimensional differential calculus, in the Bismut approach the infinite dimensional feature is absorbed by the Girsanov formula, while the differential analysis has  a strictly finite-dimensional character (see \cite{bismut1981martingales}). We remark that the finite-dimensional Malliavin calculus still retains great interest, firstly because both the finite-dimensional differential operators and the integration by parts formula  allow us to understand how to pass to the infinite-dimensional limit, and secondly because the strategy mentioned could be useful for generalising the calculus itself in other directions.\\

In this paper we propose a novel approach to finite-dimensional Malliavin-Bismut calculus starting from Lie's symmetries of a given autonomous SDE. The approach is based on the recent (stochastic) Lie's symmetries theory (see \cite{AlDeMoUg2019,AlDeMoUg2018,DMU2,DMU,DMU3,DMU4})  according to which a symmetry of a SDE is a (finite) stochastic transformation sending a solution process to another solution process to the same SDE (see \cite{Craddock2004} and \cite{Craddock2007} for an application of Lie symmetry methods for calculating an interesting class of expectations of It\^o's diffusions starting from the deterministic Lie's symmetries of the associated PDE, see also \cite{Gaeta2023} and references therein for recent developments). Such invariance property of the law of the solution process, which has a global character, allows us to directly formulate an analogue of the quasi-invariance principle of the Malliavin-Bismut calculus (Theorem \ref{fundamental_invariance}). Moreover, considering the same stochastic transformation as a perturbation of the solution process and taking a functional of the diffusion process, we derive a finite-dimensional integration by parts formula from the quasi-invariance principle (see Theorem \ref{integrationbyparts1}). Although the two results, the quasi-invariance principle and the integration by parts formula, are intimately related, the first has a very elementary formulation, while the last achievement requires much more calculations and some suitable technical conditions which are typical of a stochastic variational framework.

The main tool in Lie's symmetries theory is provided by the determining equations, whose solutions are vector fields corresponding to infinitesimal symmetries. Given a (multi-component) infinitesimal symmetry of a SDE, the natural geometrical setting allows us to introduce the one parameter group corresponding to each components, which is nothing else that the one parameter flow associated with the symmetry.
 Since, in order to obtain the integration by parts formula, we have to take suitable derivatives with respect to the flow parameters we need first to establish the existence of such flow components and then to assume all necessary analytical conditions permitting us to take the derivatives with respect to the flow parameter inside the expectation. All technical tools are faced and discussed in a dedicate section. We remark that  the notion of symmetry of a SDE, and in particular its invariance meaning, allows to notably simplify the analytical conditions (see Theorem \ref{technicalderivability} in Section 6).\\

The main advantage of our approach is that it easily includes, beside spatial and measure change transformations, also time transformations (see \cite{DMU} for the introduction of random time change and of rotational transformation for Brownian motion driven SDE, and \cite{AlDeMoUg2018} and \cite{AlDeMoUg2019} for a Lie's symmetry of more general SDEs). Indeed, while spatial transformation via diffeomorphisms gives rise to vector fields with a direct and natural interpretation as generators of the associated flow, by introducing a more rich geometrical setting both time and measure changes can be considered, even though their interpretation sometimes can appear more complex. In other words, we are able to provide a geometrical structure that unifies and puts on the same level all the main transformations of a given SDE, beyond the original change of measure transformation, and thus from this point of view our approach can be considered a generalization of the Bismut way to Malliavin calculus. Furthermore, within the large class of Girsanov measure transformations we privilege the subclass of quasi Doob change of measures because of their useful properties (\cite{DMU3},\cite{DMU4}). Indeed, stochastic models usually present a sufficient number of quasi Doob symmetries and for quasi Doob transformations the Dade-Doleans exponential reduces to a  simpler form without stochastic integrals. Finally, the identification of the one parameter group associated with a given infinitesimal symmetry of the SDE allows us to obtain, as in Bismut's calculus, completely explicit representations of the objects arising in the integration by parts formula.\\

We provide applications of our novel strategy to a one-dimensional Brownian motion, a generalization of the one-dimensional Ornstein-Ulemberg process, the one-dimensional Bessel process and finally to a family of two-dimensional  stochastic volatility models (including Heston model). In order to help the reader, in the examples  we show all calculations and we discuss in details all analytical conditions permitting us to derive the related integration by parts formulas.\\

The plan of the paper is the following. In order to sketch the classical framework, in Section 2 we recall the Bismut's fundamental quasi-invariance principle and the associated and celebrated Clark-Ocone theorem. A brief introduction to Lie's symmetries analysis of SDEs is provided in Section 3, including spatial transformations, deterministic time change and change of measure transformations of Girsanov and quasi Doob type. Section 4 contains an elementary and self-contained geometric description of the Lie's group which arises from the set of stochastic transformations introduced in the previous section. A quasi-invariance principle based on our Lie's symmetry approach to diffusion processes and a (finite-dimensional) integration by parts formula are proposed in Section 5. In Section 6 we provide a general scheme which could be useful to verify the technical conditions of the main theorems. In the last section we discuss the fruitful applications of our strategy to several stochastic models. In Appendix A a generalization of our main result, Theorem \ref{integrationbyparts1}, to the case of smooth cylindrical functionals of the process is provided.\\

\section{A brief recall of the fundamental invariance principle of Malliavin-Bismut calculus}

It is well-known that from the elementary invariance principle for the Lebesgue integral one can easily deduce the integration by part formula. Indeed, since $\forall \epsilon >0$ by denoting with $\lambda $ the usual Lebesgue measure we have
\[
\int_{-\infty}^{+\infty}f(x)d\lambda(x)=\int_{-\infty}^{+\infty}f(x+\epsilon)d\lambda(x),
\]
this implies that
\[
\int_{-\infty}^{+\infty}f^\prime(x)d\lambda(x)=0.
\]
Setting $f=gh$ we get
\[
\int_{-\infty}^{+\infty}g^\prime(x) h(x)d\lambda(x)+\int_{-\infty}^{+\infty}g(x) h^\prime(x)d\lambda(x)=0.
\]

At the basis of Malliavin-Bismut calculus lies an integration by parts formula which generalizes the one for Lebesgue measure to the case of the Wiener measure (i.e. an integration by parts formula involving functionals of a Brownian motion). For the reader's convenience, following \cite{bismut1981martingales,RoWi2000,williams1981begin}, we briefly recall (formally) the fundamental invariance principle of the Bismut approach to Malliavin calculus,  which will be adapted  to the case of symmetric SDEs in the following sections.\\

Let $(\Omega,\mathcal{F},\mathcal{F}_t,\mathbb{P})$ be a filtered probability space. We assume that $ \Omega={C}([0,\mathcal{T}])$ with a fixed time horizon $ \mathcal{T} < \infty$, $\mathbb{P} $ is the Wiener measure and ${W} $ a standard Wiener process. Let $u $ be a predictable bounded and smooth stochastic process and set $\phi(t)=\int_0^tu_sds, \forall t\in [0,\mathcal{T}] $. For $\epsilon\in \mathbb{R}_+ $ let us consider the following Doleans-Dade exponential process
\[
Z^{\epsilon}:=\exp\left(\epsilon\int_{0}^{\mathcal{T}}u_s dW_s-\frac{\epsilon^2}{2}\int_{0}^{\mathcal{T}}u_s^2ds\right).
\]
By Radon-Nykodim theorem we can introduce on $(\Omega,\mathcal{F})$ a new measure $\mathbb{Q}^{\epsilon} $, equivalent to the original one on the Brownian natural filtration, such that
\[
\frac{d\mathbb{Q}^{\epsilon}}{d\mathbb{P}}=Z^{\epsilon}.
\]
As well-known, Girsanov theorem implies that $W-\epsilon \phi$ is a $\mathbb{Q}^{\epsilon}-$ Wiener process. Since the expectation depends only on the law of the process, for all strongly differentiable function $g$ on $\Omega$ we have
\[
\mathbb{E}_{\mathbb{P}}[g(W)]=\mathbb{E}_{\mathbb{{Q}^{\epsilon}}}\left[g\left(W-\epsilon \int_0^tu_sds\right)\right].
\]
This formula states in particular that both the quantities do not depend on the parameter $\epsilon$.
By applying a Radon-Nykodim change of measure we obtain the well-known fundamental quasi-invariance principle
\begin{equation}\label{invariance_principle10}
\mathbb{E}_{\mathbb{P}}[g(W)]=\mathbb{E}_{\mathbb{{P}}}\left[g\left(W-\epsilon \int_0^tu_sds\right)\exp\left(\epsilon\int_{0}^{\mathcal{T}}u_s dW_s-\frac{\epsilon^2}{2}\int_{0}^{\mathcal{T}}u_s^2ds\right)\right].
\end{equation}
Considering the simplest case  $g(W):=g(W_{\mathcal{T}})$, where $g:\mathbb{R} \rightarrow \mathbb{R}$ is a smooth bounded function, deriving, at least formally, both sides of equation \eqref{invariance_principle10}  with respect to $\epsilon $ and evaluating at  $\epsilon=0 $ we get
\begin{equation}\label{integrationbyparts_zero}
0=-\mathbb{E}_{\mathbb{{P}}}\left[g'(W_{\mathcal{T}})\left(\int_0^tu_sds\right)\right]+\mathbb{E}_{\mathbb{{P}}}\left[g(W_{\mathcal{T}})\left(\int_{0}^{\mathcal{T}}u_s dW_s\right)\right],
\end{equation}
which gives an integration by parts formula for the function $g(W_{\mathcal{T}})$. If we take $g$ as a generic function of the Brownian motion $W$, we can replace the derivatives $g'$ with the Malliavin derivatives $D_t(g)$ obtaining the general integration by parts formula of Malliavin calculus (see, e.g., \cite{williams1981begin}).\\

The mentioned integration by parts formula permits to prove both the closability of Malliavin derivatives and, when applied to regular solutions to SDEs (i.e. under suitable conditions on the coefficients), the smoothness of their transition probabilities. A similar but slightly different approach to the same problem was proposed by Bismut in \cite{bismut1981martingales} permitting to generalize  formula \eqref{integrationbyparts_zero} to a pathwise unique strong solution $X$ of a Brownian motion driven SDE with smooth coefficients by considering the geometrical structure associated with the flow generated by the SDE itself. For convenience of the reader we sketch here the presentation given in (\cite[Chapter 3]{bell2012malliavin}) as a very short and formal introduction to Bismut work.\\

Consider the equation of the flow $\varphi_t(\omega,x)$ associated with the one dimensional SDE with (smooth) coefficient $\mu,\sigma$, i.e. the process solving the equation
\begin{equation}\label{eq:SDEflow}
\varphi_t(\omega,x)=x+\int_0^t\mu(\varphi_s(\omega,x))ds+\int_0^t\sigma(\varphi_s(\omega,x))dW_s.
\end{equation}
If $u_s$ is a predicable stochastic process which is square integrable with respect to time, we introduce the process
\[ \eta_t^{u}(x)=\partial_x\varphi_t(\omega,x)\int_0^t{\frac{\sigma(\varphi_s(\omega,x))}{\partial_x\varphi_s(\omega,x)}u_sds}. \]
It is interesting to note that
\[(\varphi^{*})^{-1}(\sigma(x)\partial_x)=\frac{\sigma(\varphi_s(\omega,x))}{\partial_x\varphi_s(\omega,x)}\partial_x,\]
where $(\varphi^{*})^{-1}(\omega,\cdot)$ is the inverse of the push-forward associated with the diffeomorfism $\varphi_t(\omega,\cdot)$.
If $X_t$ is a solution to the SDE \eqref{eq:SDEflow} (i.e. $X_t=\varphi_t(\cdot,X_0)$), then, under suitable conditions on the regularity and bounds of the coefficients $\mu$ and $\sigma$, one gets the following generalization of equation \eqref{integrationbyparts_zero} involving the process $X_t$ on the time horizon $[0,\mathcal{T}]$
\begin{equation}\label{eq:integrationbyparts_SDE}
0=-\mathbb{E}_{\mathbb{{P}}}\left[g'(X_{\mathcal{T}})\eta^u_{\mathcal{T}}\right]+\mathbb{E}_{\mathbb{{P}}}\left[g(X_{\mathcal{T}})\left(\int_{0}^{\mathcal{T}}u_s dW_s\right)\right].
\end{equation}
In the special case where $\mu=0$ and $\sigma=1$ formula \eqref{eq:integrationbyparts_SDE} coincides with formula \eqref{integrationbyparts_zero}.  Furthermore, the possibility of iterating \eqref{eq:integrationbyparts_SDE} for obtaining the higher order derivatives $g^{(n)}(X_t)$ of $g$ in terms of $g$ itself, permits to prove both the existence of a smooth density for the random variable $g(X_t)$ and a related martingale representation theorem (\cite{bismut1981martingales}).\\

One of the main differences between formula \eqref{integrationbyparts_zero} and formula \eqref{eq:integrationbyparts_SDE} is that while equation \eqref{integrationbyparts_zero} involves only a local and explicit expression of the Brownian motion $W_t$ (or, more generally, of the SDE $dX_t=dW_t$), in order to calculate the integration by parts formula \eqref{eq:integrationbyparts_SDE} we need to compute the highly non-local expression $\eta^u_t(x)$ involving the derivatives of the flow $\partial \varphi_t$. This means that we have to know the solution to the SDE \eqref{eq:SDEflow} starting from different initial spatial points $x$,  which generically cannot be expressed by a closed formula involving only the processes $X_t, W_t$ and $u_t$. In the rest of the paper we show how, with a direct generalization of the reasoning used to obtain \eqref{integrationbyparts_zero}, it is possible to simplify formula \eqref{eq:integrationbyparts_SDE} exploiting the symmetries of the SDE $(\mu,\sigma)$  and the related invariance properties. Our new direct proof of the integration by parts formula for solutions to Brownian motion driven SDEs allows us to avoid the (generically necessary) regularity and growth assumptions on the coefficients $\mu,\sigma$, essentially by using the important invariance properties of the process.

\section{A class of Lie's symmetries of a SDE \label{section_1}}

In this section we recall a particular class of symmetries of a SDE which will be used in the following. For more general classes and for all the proofs see \cite {DMU3} and \cite{DMU4}. A very general settings including SDEs driven by semimartingales with jumps can be found in \cite{AlDeMoUg2019,AlDeMoUg2018,DeMoUg2019}. \\

For simplicity, the Einstein summation convention on repeated indices is used throughout the paper. Let $M, M'$ be  open subsets of $\mathbb{R}^n$ and let us fix a finite time horizon $[0,\mathcal{T}]$ with $\mathcal{T}\geq0$.\\

In this paper we consider the following weak solutions to the class of autonomous SDEs.
\begin{definition}\label{SDE1}
	We say that a  SDE $\left(\mu,\sigma\right)$ admits a weak solution with initial distribution $\mu$
if there exists a probability space $\left(\Omega,\mathcal{F}, (\mathcal{F}_t)_{t \in [0,\mathcal{T}])} ,\mathbb{P}\right)$ satisfying the usual conditions, and a couple of semimartingales $(X,W)$ (taking values in $\Real^n $ and $\Real^m$ respectively) such that for all $t \in [0,\mathcal{T}]$,

i) $W$ is an $\mathcal{F}_t$ Brownian motion;

ii) $X_0$ has law $\mu$;

iii)
$\int_{0}^{t}|\sigma \sigma^T(X_s)|(\omega)+|\mu(X_s)|\left(\omega \right)ds < \infty \,  \operatorname{for} \, \operatorname{almost} \,  \operatorname{every} \,
\omega \in \Omega$;

iv)
	\begin{equation}\label{SDE_definition}
X_{t} -X_0=\int_{0}^{t  } \mu \left(X_s\right)ds + \int_{0}^{t }\sigma_\alpha \left(X_s\right)dW^\alpha_s	
	\end{equation}
If there exists a weak solution for each initial distribution $\mu$, then we say that there is a weak solution to \eqref{SDE_definition}. 	
	
\end{definition}
Let $\Omega^n \equiv C([0,\mathcal{T}], \Real^n) $ be the path-space and consider the filtration $\mathcal{H}^n_t=\sigma \left( X_s, s \leq t \right )$.
Suppose that $(X,W)$ is a weak solution to \eqref{SDE_definition} starting at $x\in \Real^n$ and let $\mathbb{P}_x $ be the law of $X$. We know that  $\mathbb{P}_x $ is a probability measure on $ (\Omega^n,\mathcal{H}^n) $ and that $\mathbb{P}_x $ solves the martingale problem for $(\sigma\sigma^T,\mu)$ starting at $x$, according to the following definition.

\begin{definition}[Martingale problem solution]
Let $\sigma $ and $\mu$ be previsible path functionals. Then for $x\in \Real^n$ we say that	the probability measure $\mathbb{P}_x $ is a solution to  the martingale problem for $(\sigma\sigma^T,\mu)$ starting at $x$ if the following conditions holds

i) $\mathbb{P}_x(X_0=x)=1 $;

ii) for each $g \in C^\infty(\Real^n)$, under $\mathbb{P}_x $,
\[
C^g_t=g(X_t)-g(X_0)-\int_{0}^{t}Lg(X_s)ds
\]
is an $\mathcal{H}^n_t- $ martingale, where	
\begin{equation}\label{eqGENERATORL}
L= \frac{1}{2}\left(\sigma \sigma^T\right)^{ij} \partial_i \partial_j+ \mu^i \partial_i.
\end{equation}
\end{definition}

It is well-known that the martingale-problem formulation of a SDE is equivalent to the weak-solution formulation.

\subsection{Spatial transformations via diffeomorphisms}
The most natural transformation of a SDE is a  diffeomorphism  $\Phi:M \to M'$ acting on the process component $X$.
Denoting with $\nabla \Phi : M \to \operatorname{Mat}\left(n,m\right)$ the Jacobian matrix
\[	 (\nabla \Phi)_j^i=\partial_j \Phi^i.\]
and  applying  It\^o formula (see, e.g., \cite{RoWi2000} Section 32 or \cite{Oksendal} Chapter 4) we prove the following result.
\begin{proposition}\label{Prop1}
	Given a diffeomorphism $\Phi : M \to M'$, if the process $\left(X,W\right)$ is solution to the  SDE $\left(\mu,\sigma\right)$, then the process $\left(\Phi\left(X\right),W\right)$ is solution to the SDE $\left(\mu',\sigma'\right)$ with
	
	\begin{eqnarray*}
		\mu'&=  L\left(\Phi\right)\circ \Phi^{-1}\\
		\sigma'&=\left(\nabla \Phi \cdot \sigma\right)\circ\Phi^{-1}
	\end{eqnarray*}
where $L$ is the infinitesimal generator given in \eqref{eqGENERATORL}.
\end{proposition}

\subsection{Deterministic time changes}
In this paper for simplicity we consider only deterministic time change. For random time change in Lie's symmetry framework see \cite{DMU} and \cite{AlDeMoUg2019}.
We denote by $f(t) $ a deterministic absolutely continuous time transformation given by
\[ t'=f(t)=\int_0^t f^\prime (s)ds,\]
where $\eta:=f^\prime$ is a smooth and strictly positive real function such that $f(0)=0 $. Time change is a tranformation
acting on both components of the solution process $(X,W)$ which we denote with $\mathcal{H}_\eta $.\\
If $W'$  is the solution to

\[%
dW'_t=\sqrt{\eta (t) }dW_t,
\]%
then $\mathcal{H}_\eta\left(W^\prime\right)$ is again a Brownian motion.

\begin{proposition}\label{Prop2}
	Let $\eta :\mathbb{R}_+  \to \mathbb{R}_+$ be a smooth and strictly positive function.
Let $\left(X,W\right)$ be a solution to the SDE $\left(\mu,\sigma\right)$. Then the process $\left(H_\eta\left(X\right),H_\eta \left(W'\right)\right)$ is solution to the SDE $\left(\mu',\sigma'\right)$ with
	\begin{eqnarray*}
				\mu'&=&\frac{1}{\eta}\mu\\
		\sigma'&=&\frac{1}{\sqrt{\eta}}\sigma.
	\end{eqnarray*}
\end{proposition}

Denoting by $f^{-1}(t)$ the inverse of $f(t)$, we can recall the analogue of the above proposition for the martingale problem formulation of a SDE.

\begin{proposition}[Rogers-Williams,V.26]
Let $\eta :\mathbb{R}_+  \to \mathbb{R}_+$ be a smooth and strictly positive function. Then the time change $f^{-1}(t) $ transforms the martingale problem for $ (a,\mu)$ into the martingale problem for $(\frac{a}{\eta},\frac{\mu}{\eta,})$ (with $ a=\sigma\sigma^T$) 	
\end{proposition}

\subsection{Random measure changes}

In order to further enlarge the family of admissible transformations, we randomly change the probability measure under which the driven process is a Brownian motion by using Girsanov theorem.\\
For the probabilistic theory of absolutely continuous predictable change of measure of Brownian motion the reader can see, e.g.,  \cite{RoWi2000} Section 38, while an extended treatment of the subject can be found in \cite{Oksendal} Section 8.6.\\

Given a previsible process $u_{t\in[0,\mathcal{T}]} $, let us define the Doleans-Dade exponential process  $\left(Z_t\right)_{t\in[0,\mathcal{T}]}$ by setting
\begin{equation}\label{supermartingale}
Z_t=Z_t(\theta):=\exp\left\{\int_{0}^{t}u_sdW_s -\frac{1}{2}\int_{0}^{t}u_s^2ds \right\}.
\end{equation}
By It\^o formula one has
\[ dZ_t=Z_t\theta_tdW_t,\]
which says that $Z$ is a local martingale. Indeed,  since $Z_t$ is strictly positive, we have that $Z_t$ is always a supermartingale.
Moreover, being  the Radon-Nikodym derivative   in Girsanov theorem strictly positive, i.e. $\frac{\mathbb{Q}_{\mathcal{F}_{\mathcal{T}}}}{\mathbb{P}_{\mathcal{F}_{\mathcal{T}}}}=Z_\mathcal{T}>0$, one can prove that the measure  $\mathbb{Q}_{\mathcal{F}_{\mathcal{T}}}$ is actually equivalent to  $\mathbb{P}_{\mathcal{F}_{\mathcal{T}}}$.\\

Instead of using the well known Novikov condition (see \cite{ReYo1999}, Chapter VIII, Proposition 1.14) to force the supermartingale $Z_t$ to be a $\mathbb{P}$-(global) martingale, according to \cite{DMU3} we follow an alternative strategy, given by the next Lemma, which  assumes the non explosiveness property both of the original SDE and of the transformed one, according to the following definition.

\begin{definition} Let $\mu\colon{M}\to{\mathbb{R}^n}$ and $\sigma\colon{M}\to{Mat(n,m)}$ be two smooth functions. The SDE $(\mu,\sigma)$ is called \emph{non explosive} if any solution $(X,W)$ to $(\mu,\sigma)$ is defined for all times $t\geq0$.
	
	A smooth vector field $h$ is called \emph{non explosive} for the non explosive SDE $(\mu,\sigma)$ if the SDE $(\mu+\sigma\cdot{h},\sigma)$ is a non explosive SDE.
	
	A positive smooth function $\eta$ is called a \emph{non explosive} time change for the non explosive SDE $(\mu,\sigma)$ if the SDE $\Bigr(\frac{\mu}{\eta},\frac{\sigma}{\sqrt{\eta}}\Bigl)$ is non explosive.
\end{definition}
In the following we consider a smooth function $h\colon M\to\mathbb{R}^m$ such that $h(X)$ is a predictable and non explosive stochastic process for the continuous solution process $X$ of the SDE $(\mu,\sigma)$.
\begin{lemma}[\cite{DMU3}]
	\label{lemm:lea}
	Let $(\mu,\sigma)$ be a non explosive SDE with a weak solution $(X,W)$ and let $h\colon{M}\to\mathbb{R}^n$ be a smooth non explosive vector field. Then the exponential supermartingale $(Z_t)_{t\in[0,\mathcal{T}]}$ associated with $u_t=h(X_t)$ is a $\mathbb{P}$-(global) martingale.
\end{lemma}

The following theorem shows how this probability measure change works on a given SDE.
\begin{theorem}
	\label{theo:theg}
	Let $(X,W)$ be a solution to the non explosive SDE $(\mu,\sigma)$ on the probability space $(\Omega,\mathcal{F},\mathbb{P})$ and let $h$ be a smooth non explosive vector field for $(\mu,\sigma)$. Then $(X,W')$ is a solution to the SDE $(\mu',\sigma')=(\mu+\sigma\cdot{h},\sigma)$ on the probability space $(\Omega,\mathcal{F},\mathbb{Q})$, where
	
	\begin{align*}
	W'_t &=  -\int_0^t h(X_s)ds+W_t,\\
	\left.\frac{d\mathbb{Q}}{d\mathbb{P}}\right|_{\mathcal{F}_{\mathcal{T}}} &=  \exp \left( \int_0^\mathcal{T} h_j(X_s)dW_s^j-\frac{1}{2}\int_0^\mathcal{T} \sum_{j=1}^m(h_j(X_s))^2ds \right).
	\end{align*}
	
\end{theorem}

\begin{theorem}[Rogers-Williams, V.27]
	Let us suppose that $\mathbb{P}_x $ solves the martingale problem for $(a,\mu)$ starting from $x $, that $h:\Omega \rightarrow \Real^n$ is a bounded previsible path functional and that $a=\sigma\sigma^T$ and $\mu$ are bounded. Then
	
\begin{equation}\label{supermartingale2}
\tilde{Z_t}:=\exp\left\{\int_{0}^{t}h(Xs)^TdM_s -\frac{1}{2}\int_{0}^{t}h(X_s)a(X_s)h(X_s)ds \right\}
\end{equation}	
is a $\mathbb{P}_x $ martingale with
\[
M_t=X_t-\int_{0}^{t}\mu(X_s)ds.
\]
Defining a measure $\mathbb{Q}$ on $(\Omega, \mathcal{H}) $ by
\[
\frac{d\mathbb{Q}}{d\mathbb{P}}|_{\mathcal{H}_{t}} \quad on \quad {\mathcal{H}_{t}},
\]
then $\mathbb{Q}$ solves the martingale problem for $(a,\mu^\prime)$ starting from $x $, where
\[
\mu^\prime=\mu+ah.
\]
\end{theorem}

We put together the previous stochastic transformations in the following natural way.
\begin{definition}[Stochastic transformation] Given two open subsets $M$ and $M'$ of $\mathbb{R}^n$, a diffeomorphism $\Phi \colon{M}\to{M'}$, a deterministic time change $f(t)$ and a random change of measure $h\colon{M}\to\mathbb{R}^m$, we call $T=(\Phi,\eta,h)$ a \emph{(weak finite) stochastic transformation}.
\end{definition}
In order to explicitly describe how the random transformation $T$ acts on the solution process we give the following definition.
\begin{definition}
	\label{defi:dea}
	Let $T=(\Phi,\eta,h)$ be a stochastic transformation. Let $X$ be a continuous stochastic process taking values in $M$ and $W$ be an $m$-dimensional Brownian motion in the space $(\Omega,\mathcal{F},\mathbb{P})$ such that the pair $(X,W)$ is a solution to the non explosive SDE $(\mu,\sigma)$. Given  two smooth non explosive functions $h$ and $f$ for the same SDE, we can define the process $P_T(X,W)=(P_T(X),P_T(W))$, where $P_T(X)$ takes values in $M'$ and $P_T(W)$ is a Brownian motion into the space $(\Omega,\mathcal{F},\mathbb{Q})$.
	The process components are given by
	\begin{align*}
	X'=P_T(X)&=\Phi(\mathcal{H}_{f_{-1}}), \notag \\
	W'=P_T(W)&=\tilde W, \notag \\
	\end{align*}
	where $\tilde W_t$ satisfies
	\begin{align*}
	d\tilde W_t&=\sqrt{\eta(t)}(dW_t-h(X_t)dt), \notag \\
	\end{align*}
	and
	\begin{align*}
	\frac{d\mathbb{Q}}{d\mathbb{P}}\Bigg|_{\mathcal{F}_{\mathcal{T}}}&=\exp\left(\int_0^\mathcal{T}h_j(X_s)dW_s^j-\frac{1}{2}\int_0^\mathcal{T}\sum_{j=1}^m(h_j(X_s))^2ds\right). \notag\\
	\end{align*}
	We call $P_T(X,W)$ the \emph{transformed process} of $(X,W)$ with respect to $T$ and we call the function $P_T$ the \emph{process transformation} associated with $T$.
\end{definition}

\begin{theorem}
	\label{theo:thea}
	Given a stochastic transformation $T=(\Phi,\eta,h)$ and a solution $(X,W)$ to the non explosive SDE $(\mu,\sigma)$ such that $E_T(\mu,\sigma)$ is non explosive, then $P_T(X,W)$ is solution to the SDE $E_T(\mu,\sigma)$.
\end{theorem}

\section{Geometric setting}
In this section we present the main ideas of Lie group theory underlying our set of stochastic transformations of autonomous SDEs. In this setting, the important feature
of Lie Groups is that they have the structure of a differential manifold and, in particular, their elements can vary continuously.
\begin{definition}[Lie group]
An $r$-parameters Lie Group is a group G with the structure of an $r$-dimensional manifold such that the group operation
\[
m:G\times G \rightarrow G, \quad m(g_1,g_2)=g_1 \cdot g_2, \quad g_1,g_2 \in G
\]
and the inversion
\[
i:G \rightarrow G, \quad i(g)=g^{-1}, \quad g\in G
\]
are smooth functions between manifolds.
\end{definition}
Very often, as in our case, Lie groups naturally arise as transformations groups between manifolds.
Let $G=\mathbb{R}_+\times\mathbb{R}^m$ be the group of traslations with a scaling factor, whose elements $g=(\eta,h)$ can be identified with the matrices
\[
\begin{pmatrix}
\sqrt{\eta} & h \\
0 & 1
\end{pmatrix}.
\]
Given the trivial principal bundle $\pi\colon{M\times{G}}\to{M}$ with structure group $G$, we can define the action of $G$ on $M\times{G}$ as
\begin{align}
R_{g_2}\colon{M\times{G}}&\to{M\times{G}} \notag \\
(x,g_1)&\mapsto(x,g_1\cdot{g_2}). \notag
\end{align}
which leaves $M$ invariant, with the standard product in $G$ given by $g_1\cdot{g_2}=(\eta_1,h_1)\cdot(\eta_2,h_2)=(\eta_1\eta_2,\sqrt{\eta_1}h_2+h_1)$.
Considering a second trivial principal bundle $\pi'\colon{M'\times{G}}\to{M'}$, we say that a diffeomorphism $F\colon{M\times{G}}\to{M'\times{G}}$ is an \emph{isomorphism} if $F$ preserves the structures of principal bundles of both $M\times{G}$ and $M'\times{G}$, i.e. there exists a diffeomorphism $\Phi\colon{M}\to{M'}$ such that for any $g\in{G}$
\begin{align}
\pi'\circ F&=\Phi\circ\pi, \notag \\
F\circ{R_g}&=R_g\circ{F}. \notag \\
\end{align}

We notice that such an isomorphism is completely determined by its value on $(x,e)$ (where $e$ is the unit element of $G$). Therefore there is a natural identification between a stochastic transformation $T=(\Phi,\eta,h)$ and the isomorphism $F_T$ such that $F_T(x,e)=(\Phi(x),g)$, where $g=(\eta,h)$.

The next Theorem provides the explicit form of the composition of two stochastic transformations in this group setting.
\begin{theorem}[Composition law]
Let   $T_1=(\Phi_1,\eta_1,h_1)$ and $T_2=(\Phi_2,\eta_2,h_2)$ be two stochastic transformations. Then the composition $T_2\circ T_1$ is defined as the stochastic transformation
\[
T_2\circ{T_1}=\Bigr(\Phi_2\circ\Phi_1,(\eta_2\circ\Phi_1)\eta_1,\frac{1}{\sqrt{\eta_1}}\cdot(h_2\circ\Phi_1)+h_1\Bigl),
\]
and the inverse transformation of $T=(\Phi,\eta,h)$ can be expressed as
\begin{equation}
\label{inversetransformation}
T^{-1}=\Bigr(\Phi^{-1},(\eta\circ\Phi^{-1})^{-1},-\frac{1}{\sqrt{\eta}}({h}\circ\Phi^{-1})\Bigl).
\end{equation}	
\end{theorem}

\begin{proof}
Applying the first random change of measure together with the deterministic time transformation to the Brownian motion we get
\[
dW^\prime_t=\sqrt{\eta_1}(dW_t-h_1dt);
\]
now if we apply the second transformation, we obtain
\[
dW^{\prime\prime}_t=\sqrt{\eta_2}(dW^\prime_t-h_2dt)=\sqrt{\tilde\eta}(dW_t-\tilde{\eta}dt),
\]
where $\tilde{\eta}=\eta_1\eta_2$ and $\tilde{h}=h_1+\frac{h_2}{\sqrt{\eta_1}}$.
Since after the SDE transformation by $T_1$ the state variable is $\Phi_1(X)$ and since $T_2$ acts on the new variable $\Phi_1(X)$, both $h_2$ and $\eta_2$ depend on the actual value of the process, that is $h_2(\Phi_1(X))$ and $\eta_2(\Phi_1(X))$.
In the same way we can compute the explicit form of  the inverse transformation.
\end{proof}

The following theorem shows the notable probabilistic counterpart in terms of SDE and process transformation of the above geometric identification.
\begin{theorem}
	\label{theo:tha}
	Let $T_1$ and $T_2$ be two stochastic transformations, let $(\mu,\sigma)$ be a non explosive SDE such that $E_{T_1}(\mu,\sigma)$ and $E_{T_2}(E_{T_1}(\mu,\sigma))$ are non explosive and let $(X,W)$ be a solution to the SDE $(\mu,\sigma)$ on the probability space $(\Omega,\mathcal{F},\mathbb{P})$. Then, on the probability space $(\Omega,\mathcal{F},\mathbb{Q})$, we have
	\begin{align}
	P_{T_2}(P_{T_1}(X,W))&=P_{T_2\circ{T_1}}(X,W), \notag \\
	E_{T_2}(E_{T_1}(\mu,\sigma))&=E_{T_2\circ{T_1}}(\mu,\sigma). \notag
	\end{align}
\end{theorem}

Since the set of our stochastic transformations forms  a group with respect to the composition $\circ$, one can introduce the one parameter group $T_\lambda=(\Phi_\lambda,\eta_\lambda,h_\lambda)$ and the corresponding \emph{infinitesimal (general) transformation} $V=(Y,\tau,H)$ obtained in the usual way
\begin{align}
Y(x)=&\partial_\lambda(\Phi_\lambda(x))\vert_{\lambda=0} \notag \\
\tau(x)=&\partial_\lambda(\eta_\lambda(x))\vert_{\lambda=0} \notag \\
H(x)=&\partial_\lambda(h_\lambda(x))\vert_{\lambda=0}, \notag
\end{align}
where $Y$ is a vector field on $M$, $\tau\colon{M}\to\mathbb{R}$ and $H\colon{M}\to\mathbb{R}^m$ are smooth functions. If $V$ is of the form $V=(Y,0,,0)$ we call $V$ a \emph{strong infinitesimal stochastic  transformation}.

\begin{proposition}[Flow reconstruction]\label{flow}
Let $V=(Y,\tau,H)$ be an \emph{infinitesimal stochastic transformation}. Then we can reconstruct the one parameter group $T_\lambda$ choosing $\Phi_\lambda$,  $\eta_\lambda$ and $h_\lambda $ as the one parameter solutions to the following system
\begin{align}
\partial_\lambda(\Phi_\lambda(x))=&Y(\Phi_\lambda(x)) \notag \\
\partial_\lambda(\eta_\lambda(x))=&\tau(\Phi_\lambda(x))\eta_\lambda(x) \notag \\
\partial_{\lambda}(h_{\lambda}(x))=&\frac{1}{\sqrt{\eta_{\lambda}}}{H(\Phi_{\lambda}(x))}, \notag
\end{align}
with initial condition $\Phi_0=id_M, \eta_0=1$ and $h_0(x)=0$.	
\end{proposition}
\begin{proof}
We prove only the equation for $\eta$ and $h$. The equation for $\Phi_\lambda $ is true by definition of $Y$. By the composition law in Theorem \ref{theo:tha} and by the properties of the flow we have
\begin{align}
\eta_{\lambda_1+\lambda_2}(x)=&{\eta_{\lambda_1}}(\Phi_{\lambda_2}(x))\eta_{\lambda_2}(x)\notag \\
\partial_{\lambda_1}(\eta_{\lambda_1+\lambda_2}(x))=&\partial_{\lambda_1}(\eta_{\lambda_1}(\Phi_{\lambda_2}(x))\eta_{\lambda_2}(x)) \notag \\
\partial_{\lambda_1}(\eta_{\lambda_1+\lambda_2}(x))\vert_{\lambda_1=0}=&\tau (\Phi_{\lambda_2}(x))\eta_{\lambda_2}(x), \notag
\end{align}
with initial condition $\eta_0(x)=1$

In the same way we obtains that $h_{\lambda}$ satifies
\begin{align}
h_{\lambda_1+\lambda_2}(x)=&\frac{1}{\sqrt{\eta_{\lambda_2}}}{h_{\lambda_1}}(\Phi_{\lambda_2}(x))+h_{\lambda_2}(x) \notag \\
\partial_{\lambda_1}(h_{\lambda_1+\lambda_2}(x))=&\frac{1}{\sqrt{\eta_{\lambda_2}}}\partial_{\lambda_1}(h_{\lambda_1}(\Phi_{\lambda_2}(x))) \notag \\
\partial_{\lambda_1}(h_{\lambda_1+\lambda_2}(x))\vert_{\lambda_1=0}=&\frac{1}{\sqrt{\eta_{\lambda_2}}}{H(\Phi_{\lambda_2}(x))}, \notag
\end{align}
with initial condition $h_0(x)=0$.
\end{proof}
If the time change is deterministic, as in the present paper, there is a function $m$ such that
\[
m=L(\tau).
\]
Indeed in this case time transformation has a closed form: if $f_{\lambda} $ solves the equation
\[
\partial_{\lambda} f_{\lambda}=m(f_{\lambda}),\
f_0=0
\]
then
\begin{equation}\label{deterministictimechange}
\int_{0}^{t}\eta_{\lambda}(x_s,s)ds=f_{\lambda}(X_t,t).
\end{equation}
Finally we introduce the relevant notion of symmetry of a SDE.
\begin{definition}[Finite and infinitesimal symmetry]\label{symmetry_definition} A stochastic transformation $T$ is a \emph{(finite weak) symmetry} of a non explosive SDE $(\mu,\sigma)$ if, for every solution process $(X,W)$, $P_T(X,W)$ is a solution process to the same SDE.
	An infinitesimal stochastic transformation $V$ generating a one parameter group $T_\lambda$ is called an \emph{infinitesimal (general) symmetry} of the non explosive SDE $(\mu,\sigma)$ if $T_\lambda$ is a symmetry of $(\mu,\sigma)$.
\end{definition}
\begin{proposition}
	\label{prop:pra}
	A stochastic transformation $T=(\Phi,\eta,h)$ is a symmetry of the non explosive SDE $(\mu,\sigma)$ if and only if
	\begin{align}
	\Bigr(\frac{1}{\eta}[L(\Phi)+\nabla\Phi\cdot\sigma\cdot{h}]\Bigl)\circ\Phi^{-1}&=\mu \notag \\
	\Bigr(\frac{1}{\sqrt{\eta}}\nabla\Phi\cdot\sigma \Bigl)\circ\Phi^{-1}&=\sigma \notag
	\end{align}
\end{proposition}
Next theorem provides the \emph{general determining equations} satisfied by the infinitesimal symmetries of a SDE $(\mu, \sigma)$.
\begin{theorem}[[Determing equations]\label{TheoDetEq}
	An infinitesimal stochastic transformation $V=(Y,\tau,H)$ is an infinitesimal symmetry of the non explosive SDE $(\mu,\sigma)$ if and only if $V$ generates a one parameter group defined on $M$ and the following equations hold
	\begin{align}
	\label{equa:eqa}
	Y(\mu)-L(Y)-\sigma\cdot{H}+\tau\mu&=0 \\
	\label{equa:eqb}
	[Y,\sigma]+\frac{1}{2}\tau\sigma&=0.
	\end{align}
\end{theorem}

In the following we recall the not trivial subclass of the general Girsanov transformations given by the quasi Doob transformations introduced in \cite{DMU4}. An abstract introduction of this measure change class can be found, e.g., in \cite{doob}.

\begin{definition}\label{quasidoob}[Quasi Doob transformation] Let $(\mu,\sigma)$ be a non explosive SDE and let $(X,W)$ be a solution to $(\mu,\sigma)$. Given a smooth function $h\colon M\to\mathbb{R}^m$ non explosive with respect to $(\mu,\sigma)$, we say that $h$ is a \emph{quasi Doob transformation with respect to the SDE $(\mu,\sigma)$ and characterized by the smooth function} $\mathfrak{h}\colon M\to\mathbb{R}^m$ if the measure $\mathbb{Q}$ generated by the random change of measure associated with $h$ is such that on the time horizon $[0,\mathcal{T}]$
	\[
	\frac{d\mathbb{Q}}{d\mathbb{P}}\Bigg|_{\mathcal{F}_\mathcal{T}}=\exp\left\{\mathfrak{h}(X_\mathcal{T})-\mathfrak{h}(X_0)-\int_{0}^{\mathcal{T}}G_{\mathfrak{h}}(X_s)ds\right\}
	\]
	where $G_{\mathfrak{h}}$ is a suitable $C^2(\mathbb{R}^m)$ function depending on $\mathfrak{h}$.
\end{definition}
In order to provide  conditions on $\mathfrak{h}$ and $h$ which guarantees that a random change of measure  is a quasi Doob one we can prove the following Proposition.
\begin{proposition}\label{PropDoobTransf}
	Let $h\colon M\to\mathbb{R}^m$ be a smooth function associated with a random change of measure transformation on $(\mu,\sigma)$. If $\mathfrak{h}\colon M\to\mathbb{R}^m$ is a smooth function satisfying
	\begin{align}
	\label{equa:equata}
	h_j(x)&=\sigma_j^i(x)\partial_{x^i}(\mathfrak{h})(x),\\
	\label{equa:equatb}
	\frac{1}{2}\sum_{j=1}^m(h_j(x))^2&=\frac{L(\exp(\mathfrak{h}))}{\exp(\mathfrak{h})}-L(\mathfrak{h})(x)=G_{\mathfrak{h}}(x)-L(\mathfrak{h})(x),
	\end{align}
	then $h$ is also a quasi Doob transformation.
\end{proposition}
\begin{remark}
	Since in  Proposition \ref{PropDoobTransf} equation \eqref{equa:equata} depends on $\sigma$ and equation \eqref{equa:equatb} depends, through the operator $L$, on both $\mu$ and $\sigma$, the property of a change of measure to be a quasi Doob one depends on the given SDE $(\mu,\sigma)$.
\end{remark}

If in equation \eqref{equa:equata}  both $h_\lambda$ and $\mathfrak{h}_\lambda$ depend on a parameter $\lambda$ and we take the derivative with respect to that parameter in $\lambda=0$ with initial condition $h_0=0$ and $\mathfrak{h}_0=0$, we have that there exists a function $k=\partial_\lambda (\mathfrak{h}_\lambda)\vert_{\lambda=0}$ such that
\begin{align}
\label{equa:equatc}
H_j(x)&=\sigma_j^i(x)\partial_{x^i}(k)(x).
\end{align}
In this setting we finally provide the determining equations for the infinitesimal symmetries of quasi Doob type.
\begin{theorem}\label{TheoDOOB}
	An infinitesimal stochastic transformation $V=(Y,C,\tau,H)$ ( with $H=\sigma^T\cdot\nabla k$) is a symmetry of the SDE $(\mu,\sigma)$ involving only quasi Doob transformations with respect to $(\mu,\sigma)$ if and only if $V$ generates a one parameter group of transformations such that the following equations hold
	\begin{align}
	H-\sigma^T\cdot\nabla k&=0 \notag \\
	Y(\mu)-L(Y)-\sigma\cdot\sigma^T\cdot\nabla k+\tau\mu&=0 \notag \\
	[Y,\sigma]+\frac{1}{2}\tau\sigma+\sigma\cdot C&=0 \notag
	\end{align}
\end{theorem}
In the following an infinitesimal stochastic symmetry satisfying the hypotheses of Theorem \ref{TheoDOOB}  is called a \emph{quasi Doob symmetries} for the SDE $(\mu, \sigma)$.
\begin{remark}
	If in Definition \ref{quasidoob} we take
	\[
	G_{\mathfrak{h}}(x)=0
	\]
	we obtain the well-known subclass of Doob transformations, that constitute a very important class with a deep meaning in our setting. Indeed in \cite{DMU3} it was proved that there exists  a one-to-one correspondence between infinitesimal symmetries of Doob type of an SDE and Lie point infinitesimal symmetries of the associated Kolmogorov equation. As stressed in \cite{DMU3}, this fact shows that  generally the family of symmetries of an SDE is wider than the family  of the symmetries for the corresponding (deterministic) Kolmogorov equation.
\end{remark}

\section{A quasi-invariance principle and integration by parts formulas}

In this section we provide a quasi-invariance property for the solution process to a given SDE with respect to a Lie's (finite) symmetry.
More precisely we establish the quasi-invariance of the corresponding law  $\mathbb{P}$ under the class of symmetries introduced in the previous section.
Furthermore, starting from the quasi-invariance principle we derive some integration by parts formulas which are the main results of the present paper.

The next result is a fundamental invariance principle of the type of Bismut-Malliavin calculus based on the theory of Lie symmetries for a process which is solution to an autonomous SDE.
\begin{theorem}\label{fundamental_invariance}
	Let $(X,W)$ be a (weak) solution to a non explosive SDE$(\mu,\sigma)$ of the type \eqref{SDE_definition} and let $(X^\prime,W^\prime)$ be the transformed process through the (finite) stochastic transformation $T=(\Phi,\eta,h)$. Let $h$ be a predictable stochastic process which is a non explosive vector field for $(\mu,\sigma)$. If $T$ is a (finite) symmetry for the  SDE $(\mu,\sigma)$, then for any fixed $t \in [0,\mathcal{T}] $ and for  any bounded measurable function $g$,  the following quasi-invariance principle holds
	\begin{equation}\label{invariance_principle2}
	\mathbb{E}_{\mathbb{P}}[g(X)]=\mathbb{E}_{\mathbb{{P}}}\left[g(X^\prime)\exp\left(\int_{0}^{t}h_s dW_s-\frac{1}{2}\int_{0}^{t}h_s^2ds\right)\right].
	\end{equation}
\end{theorem}
\begin{proof}
	Since $T$ is a symmetry, the transformed  process $(X^\prime,W^\prime)$ is solution to the same SDE $(\mu,\sigma)$ and in particular $X^\prime$ under the measure $\mathbb{Q}$ has the same law of $X$ under $\mathbb{P}$, which implies
	\[
	\mathbb{E}_{\mathbb{P}}[g(X)]=\mathbb{E}_{\mathbb{Q}}[g(X^\prime)].
	\]
	By applying a Radon-Nykodim change of measure in the expectation with respect to the measure $\mathbb{Q} $ we obtain that for any $t \in [0,\mathcal{T}]$
	\begin{equation}\label{invariance_principle}
	\mathbb{E}_{\mathbb{P}}[g(X)]=\mathbb{E}_{\mathbb{{P}}}\left[g(X^\prime)\exp\left(\int_{0}^{t}h_s dW_s-\frac{1}{2}\int_{0}^{t} h_s^2ds\right)\right],
	\end{equation}
	which is the statement of the theorem. 	
\end{proof}

In the Malliavin-Bismut calculus it is well known that from the fundamental invariance principle an integration by parts formula can be derived.
In this section we provide some integration by parts formulas using a Lie's symmetries approach in a finite dimensional setting and, for simplicity, we provide the proof for the case of functionals of the process valuated at a single time. The extension of our integration by parts formula to smooth cylinder functionals of diffusion processes is given in Appendix A.\\

In order to prove our main results, we need the following technical lemma, that states a condition under which we can take the derivative with respect to a parameter inside the expectation.

\begin{lemma}\label{lemma_derivative2}
	Let $g(\lambda, X): \Real \times M \rightarrow \Real$ be a function which is integrable for all $\lambda$ and continuously differentiable with respect to $\lambda$ and  such that
	$\mathbb{E}_{\mathbb{{P}}}[|\partial^2_{\lambda}{g}(\lambda, X)|]<+\infty $.
Then
	\[
	\partial_{\lambda}\mathbb{E}_{\mathbb{{P}}}[{g}(\lambda, X)]=\mathbb{E}_{\mathbb{{P}}}[\partial_{\lambda}{g}(\lambda, X)].
	\]
\end{lemma}
\begin{proof}
	This can be easily proved by mean-value theorem which states that there exists  $\tau(\epsilon)\in (\lambda, \lambda+\epsilon)$ such that
	 \[
	 \partial_{\lambda}\mathbb{E}_{\mathbb{{P}}}[\tilde{g}(\lambda, X)]=\lim_{\epsilon \rightarrow 0}\mathbb{E}_{\mathbb{{P}}}\left[\frac{\tilde{g}(\lambda+\epsilon, X)-\tilde{g}(\lambda, X)}{\epsilon}\right]=\lim_{\epsilon \rightarrow 0}\mathbb{E}_{\mathbb{{P}}}[\partial_{\lambda}\tilde{g}(\tau(\epsilon), X)].
	 \]
By putting
\[
\partial_{\lambda}{g}(\tau(\epsilon), X)=\partial_{\lambda}{g}(\lambda, X)+\int_{\lambda}^{\lambda+\epsilon}\partial^2_{\lambda}{g}(\sigma^\prime, X)d\sigma^\prime,
\]
since $\mathbb{E}_{\mathbb{{P}}}[|\partial^2_{\lambda}{g}(\lambda, X)|]<+\infty $ we get the result.	
\end{proof}
For later use, in the following we recall a result about the continuity and the differentiability with respect to parameters of stochastic integrals that can be found in \cite{Kunita_flows} (see Proposition 2.3.1 for details).
\begin{lemma}[Kunita]\label{lemma-derivative3}
Let $\Sigma=\Real^e $ be the parameter set and let $p>\min ( e,  2)$. Suppose that
\begin{equation}\label{Kunita_conditions}
\sup_{\lambda}\mathbb{E}\left[\int_{0}^{T}|g_\lambda (r)|^pdr\right]<+\infty, \quad \mathbb{E}\left[\int_{0}^{T}|g_\lambda (r)-g_{\lambda'}|^pdr\right]<c_p |\lambda-\lambda'|^p
\end{equation}
where  $g_\lambda (r) $ is $n$-times continuously differentiable with respect to $ \lambda$ and, for $|i|\leq n $, the derivatives $\partial^i_\lambda g_\lambda (r)   $ satisfy \eqref{Kunita_conditions}. Then the family of It\^o's stochastic integrals $( \int_{0}^{t}g_\lambda (r)dW_r )_t $ has a modification which is $n$-times continuously differentiable with respect to $\lambda $ a.s. For any
$|i| \le n $, we have
\[
\partial^i_\lambda \int_{0}^{t}g_\lambda (r)dW_r=\int_{0}^{t}\partial^i_\lambda g_\lambda (r)dW_r \quad \forall (t,\lambda)\quad  a.s.
\]
	
\end{lemma}

We state our integration-by-parts formulas, which is the main results of our novel approach, under the following assumption.

\begin{hypothesis}\label{Hp:A}
For any solution $X_t$ to the SDE $(\mu,\sigma)$ with deterministic initial condition we have that:
	\[H_{\alpha} (X_t, t), Y (H_{\alpha}) (X_t, t), L (Y^i) (X_t, t),
	\Sigma_{\alpha} (Y^i) (X_t, t), L (Y (Y^i)) (X_t, t), \Sigma_{\alpha} (Y
	(Y^i)) (X_t, t) \in L^2 (\Omega)\]
\end{hypothesis}

\begin{theorem}[Integration by parts formula]\label{integrationbyparts1}
	Let $(X,W)$ be a solution to the SDE \eqref{SDE_definition} and let $(Y, \tau ,H)$ be an infinitesimal stochastic symmetry for the SDE according with Definition \ref{symmetry_definition}. Let us consider the one parameter group  $T_\lambda=(\Phi_{\lambda}, \eta_\lambda, h_\lambda) $ associated with the infinitesimal symmetry according to Proposition \ref{flow}. 
	Assuming Hypothesis \ref{Hp:A},
	then the following integration by parts formula holds for every $t \in [0,\mathcal{T}]$
	\begin{multline}
	- m(t)\mathbb{E}_{\mathbb{{P}}} \left[L({F})(X_t)\right]+\mathbb{E}_{\mathbb{{P}}}\left[F(X_t)\int_{0}^{t}(H_\alpha (X_s)dW^\alpha_s)\right]+
	\mathbb{E}_{\mathbb{{P}}}[(Y({F}))(X_t)] - \mathbb{E}_{\mathbb{P}} [Y (F) (X_0)]=0
	\end{multline}
	where $F$ is a bounded functional with bounded first and second derivative.
\end{theorem}

\begin{proof}

Since the random transformation $T_\lambda$ is a symmetry for the SDE, denoting by $X^\lambda_t=P_{T_\lambda}(X,W)$ the related transformed process, by definition of symmetry we have $\forall t \in [0,\mathcal{T}] $
\[
\mathbb{E}_{\mathbb{{P}}}[{F}(X_t)]=\mathbb{E}_{\mathbb{{Q}_\lambda}}[{F}(X^\lambda_t)].
\]
By It\^o formula and since $X^\lambda_t$ is solution to the same original SDE
\begin{align*}
\mathbb{E}_{\mathbb{{Q}_\lambda}}[{F}(X^\lambda_t)]=&\mathbb{E}_{\mathbb{{Q}_\lambda}}\left[\int_{0}^{t}L(F)(X^{\lambda}_s)ds_\lambda +\partial_{i} (F)(X^{\lambda}_s)\sigma_\alpha^i dW^{\alpha,\lambda}_s)\right]\\
=&\mathbb{E}_{\mathbb{{Q}_\lambda}}\left[\int_{0}^{t}L(F)(X^{\lambda}_s)ds_\lambda\right],
\end{align*}
where we used that, under the hypotheses of the theorem, the expectation of the stochastic integral with respect to $W^{\alpha,\lambda}_t$ is zero. Indeed, introducing $g(s,X_s)=\partial_i (F\circ \Phi_{\lambda})(X_s)\sigma_\alpha^i$, if $g $ is $\mathcal{B}\otimes \mathcal{F}$-measurable, $\forall s, g(s,\cdot) $ is $ \mathcal{F}$-measurable and, for any $s$,
\[
g(s,\cdot) \in L^2(\Omega, \mathcal{F}, \mathbb{Q}_\lambda),\quad \mathbb{E}_{\mathbb{{Q}}_\lambda}\left[\int_{0}^{t}|g(s,\cdot)|^2ds\right] < +\infty,
\]
then the stochastic integral with respect to the $\mathbb{Q}_{\lambda}$-Brownian motion $W^{\alpha,\lambda}_t$ is a (global) martingale. By Hypothesis A the integrand is in $L^2$ and so the It\^o integral is well-defined.

Denoting by $f_{-\lambda}(t) $ the inverse of the deterministic time-change $f_\lambda (t) $ and  applying a deterministic change of variables in the integral we get $\forall t \in [0,\mathcal{T}] $
\[
\mathbb{E}_{\mathbb{{Q}_\lambda}}\left[\int_{0}^{t}L(F)(X^{\lambda}_{s_\lambda})ds_\lambda\right]=\mathbb{E}_{\mathbb{{Q}_\lambda}}\left[\int_{0}^{f_{-\lambda}(t) } L(F)(\Phi_{\lambda}(X_s,s))f^\prime_\lambda(s)ds\right].
\]
By performing a Radom-Nikodym measure change the right-hand side expectation becomes
\[
\mathbb{E}_{\mathbb{{Q}_\lambda}}\left[\int_{0}^{f_{-\lambda}(t) } L(F)(\Phi_{\lambda}(X_s,s))f^\prime_\lambda(s)ds\right] =\mathbb{E}_{\mathbb{{P}}}\left[\left.\frac{d\mathbb{Q}_{\lambda}}{d\mathbb{P}}\right|_{\mathcal{F}_\mathcal{T}} \int_{0}^{f_{-\lambda}(t) } L(F) \circ \Phi_\lambda(X_s,s)f^\prime_\lambda(s)ds\right],
\]
so we finally get
\begin{equation}\label{IP3}
\mathbb{E}_{\mathbb{{P}}}[{F}(X_t)]=\mathbb{E}_{\mathbb{{P}}}\left[\left.\frac{d\mathbb{Q}_{\lambda}}{d\mathbb{P}}\right|_{\mathcal{F}_\mathcal{T}} \int_{0}^{f_{-\lambda}(t) } L(F) \circ \Phi_\lambda(X_s,s)f^\prime_\lambda(s)ds\right].
\end{equation}

We want to take the derivative with respect to the parameter $\lambda$ of both sides in the previous equality. Since there is no dependence on the parameter $\lambda$ in the left part, the same is true for the term in the right part.
 Here we have to take the derivative inside the expectation.
Denoting by $P(\lambda) = Z_\lambda \int_{0}^{{f_{-\lambda}(t) }}L(F) \circ \Phi_{\lambda})(X_s,s)f^\prime_\lambda(s)ds$,  where $Z_\lambda=\frac{d\mathbb{Q}_{\lambda}}{d\mathbb{P}}$, under Hypothesis A by Theorem \ref{technicalderivability}
 we have
 $\forall \lambda$
\begin{equation}\label{derivability_conditions}
\partial^2_{\lambda}P(\lambda)\in L^2,
\end{equation}
and therefore by Lemma \ref{lemma_derivative2}
\[
\partial_{\lambda}\mathbb{E}_{\mathbb{{P}}}[P(\lambda)]=\mathbb{E}_{\mathbb{{P}}}[\partial_{\lambda}P(\lambda)].
\]
Taking the derivative with respect to $\lambda $ and evaluating the result at $\lambda=0 $ in \eqref{IP3} we obtain
\begin{multline}\label{IP4}
0=\mathbb{E}_{\mathbb{{P}}}\left[\left(\int_{0}^{\mathcal{T}}H_\alpha (X_t)dW^\alpha_t\right)\left(\int_{0}^{t}L({F})(X_s)ds\right)]-m(t)\mathbb{E}_{\mathbb{{P}}}[L({F})(X_t)\right]+ \mathbb{E}_{\mathbb{{P}}}\left[\int_{0}^{t}Y(L({F}))(X_s)ds\right]+\\
+\mathbb{E}_{\mathbb{{P}}}\left[\int_{0}^{t}\tau(s) L({F})(X_s)ds\right]
\end{multline}
where the first term has been obtained by taking the partial derivative of the Doleans-Dade exponential process as in Definition \ref{defi:dea} and using the fact that, by Proposition \ref{flow},  $\partial_{\lambda}h_{\lambda, j}(x)|_{\lambda=0}=H_j(x)$, the second term by using the fundamental theorem of calculus for deriving the $\lambda $ dependent estremes of integration and the fact that $\partial_{\lambda} f_{-\lambda}(t)|_{\lambda=0}=-m(t) $ and $f^\prime_{\lambda}(t)=\eta_{\lambda}(t) $ as in \eqref{deterministictimechange}, the third one by deriving the spatial flow $\Phi_\lambda$ according with Proposition \ref{flow} together with the property that $\Phi_0$ is the identity function and the last one by observing that, since $f^\prime=\eta $, then $\partial_{\lambda}f^\prime_\lambda=\tau f^\prime_\lambda $ (with $f^\prime_0=1$).

Since for an infinitesimal symmetry we have $[Y,L]=-\tau L+H_\alpha \Sigma_{\alpha} $, this implies that
\begin{align*}
\mathbb{E}_{\mathbb{{P}}}\left[\int_{0}^{t}Y(L({F}))(X_s)ds\right]=&\mathbb{E}_{\mathbb{{P}}}\left[\int_{0}^{t}L(Y({F}))(X_s)ds\right]- \mathbb{E}_{\mathbb{{P}}}\left[\int_{0}^{t}\tau (s) L({F})(X_s)ds\right]+\\
&+ \mathbb{E}_{\mathbb{{P}}}\left[\int_{0}^{t}H_\alpha (X_s) \sigma_{\alpha}  \nabla F(X_s)ds\right].
\end{align*}
By It\^o formula we have
\begin{align*}
\mathbb{E}_{\mathbb{{P}}}\left[(Y({F}))(X_{t})\right]=&\mathbb{E}_{\mathbb{{P}}}\left[(Y({F})(X_{0})\right]\\
&+\mathbb{E}_{\mathbb{P}}\left[\int_{0}^{t}L(Y({F}))(X_s)ds\right] + \mathbb{E}_{\mathbb{{P}}}\left[\int_{0}^{t} \partial_{i} Y({F})(X_s)\sigma^i_\alpha dW^\alpha_s\right],
\end{align*}
from which we get
\[
\mathbb{E}_{\mathbb{{P}}}\left[(Y({F}))(X_{t})\right]-\mathbb{E}_{\mathbb{{P}}}[(Y({F}))(X_{0})]=
\mathbb{E}_{\mathbb{P}}\left[\int_{0}^{t}L(Y({F}))(X_s)ds\right].
\]
Let us consider the first addendum in the right-hand side of \eqref{IP4}. Since stochastic integrals are martingales we get
\[
\mathbb{E}_{\mathbb{{P}}}\left[\left(\int_{0}^{\mathcal{T}}H_\alpha (X_t)dW^\alpha_t\right)\left(\int_{0}^{t}L({F})(X_s)ds\right)\right]=\mathbb{E}_{\mathbb{{P}}}\left[\left(\int_{0}^{t} H_\alpha (X_s)dW^\alpha_s \right) \left(\int_{0}^{t}L({F})(X_s)ds\right)\right]
\]
and by integration by parts
\[ \left(\int_{0}^{t} H_\alpha (X_s)dW^\alpha_s \right) \cdot \left(\int_{0}^{t}  L(F) (X_s) ds\right) =\int_0^t \left\{ \left( \int_0^s H_\alpha (X_\tau)dW^\alpha_\tau \right) L(F) (X_s) ds \right\}  \]
\[ +\int_0^t \left(\int_{0}^{s} L(F) (X_\tau)d\tau\right) H_\alpha (X_s)dW^\alpha_s  +\left [\int_{0}^{t} H_\alpha (X_s)dW^\alpha_s, \int_{0}^{t}  L(F) (X_s) ds \right ]  \]
where the last term is a zero quadratic variation since the second term is absolutely continuous.
Since stochastic integrals are martingales
\[ \mathbb{E}_{\mathbb{P}} \left[ \int_0^t \left( \int_0^s H_\alpha (X_\tau)dW^\alpha_\tau \right) L (F) (X_s)  ds \right] = \]
\begin{equation} \label{scomposition}
 =\mathbb{E}_{\mathbb{P}} \left[ \int_0^t \left\{ \left( \int_0^s H_\alpha (X_\tau)dW^\alpha_\tau \right) L (F) (X_s) ds + \left( \int_0^s H_\alpha (X_\tau)dW^\alpha_\tau \right) \sigma_\alpha \nabla F (X_s) dW^\alpha_s
\right\} \right].
\end{equation}
Applying integration by parts formula to the second term in the right-hand side of \eqref{scomposition} we obtain
\[ \int_0^t  \left( \int_0^s H_\alpha (X_\tau)dW^\alpha_\tau \right)
\sigma_\alpha \nabla F (X_s) dW^\alpha_s = \left(\int_{0}^{t} H_\alpha (X_s)dW^\alpha_s \right) \left(\int_{0}^{t} \sigma_\alpha \nabla F (X_s) dW^\alpha_s\right) \]
\[ -\int_0^t (\int_{0}^{s}\sigma_\alpha \nabla F (X_\tau) dW^\alpha_\tau) H_\alpha (X_s)dW^\alpha_s   -\left [\int_{0}^{t} H_\alpha (X_s)dW^\alpha_s, \int_{0}^{t}  \sigma_\alpha \nabla F (X_s) dW^\alpha_s  \right ],  \]
with the quadratic variation equal to
\[  \int_0^t (\sigma^i_\alpha \partial_i F (X_s)) \cdot H_\alpha (X_s) ds. \]
Summing the two stochastic expressions for the two stochastic integrals in \eqref{scomposition}  we get
\[ \left( \int_0^t L (F) (X_s) ds + \int_0^t \sigma_\alpha \nabla F (X_s)
dW^\alpha_s \right) \left( \int_0^t H_\alpha (X_s)dW^\alpha_s
\right)  \]
\[ - \int_0^t \left( \int_0^s L (F) (X_\tau) d\tau + \int_0^s \sigma_\alpha \nabla F
(X_\tau) dW^\alpha_\tau  \right) H_\alpha (X_s)dW^\alpha_s + \]
\[ - \int_0^t (\sigma_\alpha \nabla F (X_s)) \cdot H_\alpha (X_s) ds. \]
Using the following well-known property
\[ \mathbb{E}_{\mathbb{P}} \left[ \int_0^t \left( \int_0^s L (F) (X_\tau) d\tau
 + \int_0^s \sigma_\alpha \nabla F (X_\tau) dW^\alpha_\tau \right) H_\alpha (X_s)dW^\alpha_s \right] = 0 \]
and inserting all the final expressions in \eqref{IP4}, we finally get our integration by parts formula:
\[- m (t) \mathbb{E}_{\mathbb{P}} [L (F) (X_t)] +\mathbb{E}_{\mathbb{P}} \left[
F (X_t) \left( \int_0^t H_\alpha (X_s) dW^\alpha_s \right) \right]
+\mathbb{E}_{\mathbb{P}} [Y (F) (X_t)]-\mathbb{E}_{\mathbb{P}} [Y (F) (X_0)]=0.
\]
\end{proof}

When the Lie's symmetries are of quasi Doob type,  the integration by parts formula becomes simpler and the previous result admits an interesting corollary.
\begin{hypothesis}\label{Hp:B}
	For any solution $X_t$ to the SDE $(\mu,\sigma)$ with deterministic initial condition we have that:
\[ Y (H_{\alpha}) (X_t, t), L (Y^i) (X_t, t),
	\Sigma_{\alpha} (Y^i) (X_t, t), L (Y (Y^i)) (X_t, t), \Sigma_{\alpha} (Y
	(Y^i)) (X_t, t) \in L^2 (\Omega)\]
\end{hypothesis}

\begin{corollary}\label{corollary_integration}
Let $(X.W)$ be a solution to the SDE \eqref{SDE1} and let $(Y, \tau ,H)$ be an infinitesimal stochastic symmetry of quasi Doob type for the SDE. Let us consider the one parameter group  $T_\lambda=(\Phi_{\lambda}, \eta_\lambda, h_\lambda) $ associated with the infinitesimal symmetry according to Proposition \ref{flow}. 
Assuming Hypothesis \eqref{Hp:B}
 the following integration by parts formula holds for every $t \in [0,\mathcal{T}]$
\begin{multline}
- m(t)\mathbb{E}_{\mathbb{{P}}}\left[L({F})(X_t)\right]+\mathbb{E}_{\mathbb{{P}}}\left[F(X_t)\left(k(X_t)-k(X_0)-\int_{0}^{t}L(k(X_s))ds\right)\right]+
\mathbb{E}_{\mathbb{{P}}}[(Y({F}))(X_t)]+\\ -  \mathbb{E}_{\mathbb{P}} [Y (F) (X_0)]=0
\end{multline}	
\end{corollary}
\begin{proof}
In the proof of Theorem \ref{integrationbyparts1} the only derivative with respect to $\lambda $ which changes in the first term on the righ-hand side of \eqref{IP3} is the derivative of the Doleans-Dade exponential process as in Definition \ref{quasidoob}, that in the case of quasi Doob symmetries is
\[
\partial_{\lambda}\left(\frac{d\mathbb{Q}}{d\mathbb{P}}\Bigg|_{\mathcal{F}_\mathcal{T}}\right)|_{\lambda=0}=\{\partial_{\lambda}\mathfrak{h}(X_\mathcal{T})|_{\lambda=0}-\partial_{\lambda}\mathfrak{h}(X_0)|_{\lambda=0}-\int_{0}^{\mathcal{T}}\partial_{\lambda} G_{\mathfrak{h}}(X_s)|_{\lambda=0}ds\}.
\]
Since $\partial_{\lambda}\mathfrak{h}(x)|_{\lambda=0}=k(x)$ and $\partial_{\lambda} G_{\mathfrak{h}}(x)|_{\lambda=0}=L(k(x)) $ we have

\[
-m(t)\mathbb{E}_{\mathbb{{P}}}[L({F})(X_t)]+\mathbb{E}_{\mathbb{{P}}}[\left(k(X_{\mathcal{T}})-k(X_0)-\int_{0}^{\mathcal{T}}L(k(X_\tau))d\tau\right) \int_{0}^{t}L({F})(X_s)ds]+
\]
\[
+\mathbb{E}_{\mathbb{{P}}}[\int_{0}^{t}Y(L({F})))(X_s)ds-\mathbb{E}_{\mathbb{P}} [Y (F) (X_0)]=0.
\]
Since $\left(k(X_{\mathcal{T}})-k(X_0)-\int_{0}^{\mathcal{T}}L(k(X_\tau))d\tau\right) $ is
a martingale and integrating by parts we get

\[
\mathbb{E}_{\mathbb{{P}}}[\left(k(X_{\mathcal{T}})-k(X_0)-\int_{0}^{\mathcal{T}}L(k(X_\tau))d\tau\right) \int_{0}^{t}L({F})(X_s)ds]=
\mathbb{E}_{\mathbb{{P}}}\left[F(X_t)\left(k(X_t)-k(X_0)-\int_{0}^{t}L(k(X_s))ds\right)  \right].
\]
\end{proof}

\section{Some technical considerations}

In this section we provide some technical results necessary to prove our integration by parts formula. In particular, after giving a non-explosion result for the solution to a SDE, we address the problem of how we can verify the conditions that allow us to take the derivative with respect to $\lambda$ in our geometric-analytic setting.
Since one of the key ingredients is the theory of Lyapunov functions, we begin by recalling some important facts about this topic.
\vskip 0.5 cm
{\bf Lyapunov condition}: Given an infinitesimal generator $L$ of the form \eqref{eqGENERATORL}, there exists $ \varphi \in C^{1,2}([0,T]\times \Real^n)$ such that $ \varphi \geq 0 $ and
\[
\lim_{r \rightarrow \infty} (\inf_{0 \leq t \leq T}\varphi(t,x))=\infty
\]
and, for some constant $M$, a.e. on $[0,T]\times \Real^n$
\[
(\partial_{t}+L)\varphi (t,x) \leq M \varphi(t,x).
\]
\vskip 0.5 cm
The function $\varphi$ in the above Lyapunov condition is called a Lyapunov function.
\vskip 0.5 cm
Let $\chi= \inf \left\{t \geq 0: X_t \notin \R^n \right\} $ be the life-time of the diffusion process. Non-explosion results for a diffusion process by means of Lyapunov functions are very interesting and can be stated with a linear probabilistic proof which we report here for completeness (see, e.g., \cite{lee2020analytic,StroockVaradhan}).

\begin{theorem}\label{Lyapunov_function_theorem}
	Let $a_{i.j}:=(\sigma \sigma^T)^{ij} \in H^{1,p}_{loc}(\Real^n)\cup C(\Real^n)$ and $\mu \in L^p_{loc}(\Real^n,\Real^n), p\in (n,\infty)$. A sufficient condition for the martingale problem associated with the infinitesimal generator $L $ (\eqref{eqGENERATORL}) to be well-posed is that, for each $T>0 $, there exist a number $M=M_T>0 $ and a non-negative function $ \varphi \in C^{1,2}([0,T]\times \Real^n ) $ such that $\varphi$ is a Lyapunov function on $[0,T]\times \Real^n  $. In this case the solution process is non explosive and, for any $(t,x) \in [0,T]\times\Real^n $, it holds
	\[
	\mathbb{E}_x[\varphi(t,X_t)]\leq \exp(Mt)\varphi(0,X_0), \quad t \geq 0.
	\]
\end{theorem}
\begin{proof}
	Set $\tau_r=inf\left\{ t>0: X_t \in \R^n| B_r\right\} $ where $B_r$ denotes the ball of radius $r$. Applying It\^o formula we get
	\begin{align*}
	\exp(-Mt) \varphi (t, X_{\min (t, \tau_r)})=&\varphi(x)+\int_{0}^{t}\exp(-Ms) \partial_i \varphi(s,X_s)\sigma^i_j(X_s)dW^j_s \\
	&+\int_{0}^{t }\exp(-Ms)(-M\varphi \partial_{s}\varphi + L\varphi)(s,X_s)ds
	\end{align*}
	Since, by hypothesis on $[0,T]\times \Real^n  $
	\[
	\partial_{t}\phi(t,x)+L\phi (t,x) - M \phi(t,x)\leq 0,
	\]
$	(\exp(-Mt) \varphi (t, X_{\min (t, \tau_r)})) $ is a supermartingale, which implies
	\begin{align*}
	\varphi(x)\leq& \mathbb{E}_x[\exp(-Mt) \varphi (t, X_{\min (t, \tau_r)})]\leq
	\mathbb{E}[\exp(-Mt) \varphi (t, X_{\min (t, \tau_r)}){1}_{\tau_r\leq t}]\\
	\leq& \exp(-Mt) \varphi
	 \inf_{t \in [0, T]} \mathbb{P}[\tau_r\leq t].
	\end{align*}
	Since $\mathbb{P}[\lim_{r \rightarrow \infty}\tau_r=\chi]=1  $ (see Lemma 3.17 in \cite{lee2020analytic}),  for every fixed t
	\[
	\mathbb{P}_x[\chi \leq t]=\lim_{r \rightarrow \infty}\mathbb{P}_x[\tau_r \leq t]\leq \lim_{r}\frac{\exp (Mt)\varphi(x)}{\inf_{t \in [0, T]}\varphi}=0.
	\]
	Sending $t \rightarrow \infty$  the non-explosivity follows. Furthermore, since $\tau_r \rightarrow \chi $ a.s.
	\begin{align*}
	\mathbb{E}_x[\exp(-Mt) \varphi (t, X_{ t})]=&\mathbb{E}_x[\liminf_r \exp(-Mt) \varphi (t, X_{\min (t, \tau_r)})]\\
	\leq& \liminf_r \mathbb{E}_x[ \exp(-Mt) \varphi (t, X_{\min (t, \tau_r)})]\leq \varphi(x),
	\end{align*}
	where Fatou lemma and supermartingale property has been used.
\end{proof}
	
Let us now consider the problem of how we can verify the conditions that allow us to take the derivative with respect to $\lambda$ in our geometric-analytic setting.
Given  the SDE $(\mu, \sigma)$, we write
	\[ \Sigma_{\alpha} = \sum_{j = 1}^n \sigma_{\alpha}^j (x, z) \partial_{x^j}, \qquad \mu=\mu^i\partial{x^i} \]
	and the operator
	\[ L = \partial_z + \sum_{\alpha = 1}^m \sum_{j, i = 1}^n \sigma_{\alpha}^i
	(x, z) \sigma_{\alpha}^j (x, z) \partial_{x^i x^j} + \sum_{j = 1}^n \mu^j
	(x, z) \partial_{x^j} . \]
	
	We recall that if $(Y, \tau, H)$ is a symmetry of the SDE $(\mu, \sigma)$ we
	have
	\[ [Y, \Sigma_{\alpha}] = - \frac{1}{2} \tau \Sigma_{\alpha} \]
	and
	\[ [Y, L] - \sum_{\alpha = 1}^m H_{\alpha} \Sigma_{\alpha} = - \tau L. \]
	\begin{theorem}\label{Paola}
		Let $(\mu, \sigma)$ be a SDE with symmetry $(Y, \tau, H)$. Then, for any $f
		\in C^2 (\mathbb{R} \times \mathbb{R}^n)$, we have
		\[ \left( \frac{1}{\sqrt{\eta_{\lambda}}} \Sigma_{\alpha} (f \circ
		\Phi_{\lambda}) \right) = (\Sigma_{\alpha} (f)) \circ \Phi_{\lambda} \]
		\[ \left( \frac{1}{\eta_{\lambda}} L (f \circ \Phi_{\lambda}) \right) = \mu
		(f) \circ \Phi_{\lambda} - \frac{1}{\sqrt{\eta_{\lambda}}} \sum_{\alpha
			= 1}^m h_{\lambda, \alpha}  (\Sigma_{\alpha} (f) \circ \Phi_{\lambda}). \]
	\end{theorem}
	
	\begin{proof}
		Using Proposition \ref{prop:pra} we have that, for any $f \in C^2 (\mathbb{R} \times \mathbb{R}^n)$,
		\[\left( \sigma^i_{\alpha}\frac{\partial f}{\partial x^i}\right) \circ \Phi_{\lambda} = \frac{1}{\sqrt{\eta_{\lambda}}} \left( \sigma_{\alpha}^i \frac{\partial \Phi^k_ {\lambda}}{\partial x^i}\frac{\partial f}{\partial x^k} \right) \]
		and rewriting this expression in terms of the vector fields $\Sigma_{\alpha}$ and $\mu$ we get
		\[  (\Sigma_{\alpha} (f)) \circ \Phi_{\lambda} =  \frac{1}{\sqrt{\eta_{\lambda}}} \Sigma_{\alpha} (f \circ
		\Phi_{\lambda}). \]
		In order to prove the second condition we start again from Proposition \ref{prop:pra}  which ensures that, for any $f
		\in C^2 (\mathbb{R} \times \mathbb{R}^n)$, we have
		\[  \frac{1}{\eta_{\lambda}}\left[  L (f\circ \Phi_{\lambda}) + \sum_{\alpha=1}^m h_{\lambda, \alpha }\sigma_{\alpha}^i \frac{\partial \Phi^k_{\lambda}}{\partial x^i}\frac{\partial f}{\partial x^k}\right]=\left( \mu^i\frac{\partial f}{\partial x^i} \right)\circ \Phi_{\lambda}.
		\]
		When we write the previous expression in terms of the vector fields $\Sigma_{\alpha}$ and $\mu$ we get
		\[\frac{1}{\eta_{\lambda}}\left[  L (f\circ \Phi_{\lambda}) + \sum_{\alpha=1}^m h_{\lambda, \alpha }\Sigma_{\alpha}(f\circ  \Phi_{\lambda}) \right]= \mu(f) \circ \Phi_{\lambda} \]
		and using the relation between $\Sigma_{\alpha}(f\circ  \Phi_{\lambda})$ and $\Sigma_{\alpha}(f)\circ  \Phi_{\lambda}$ we have
		\[  \frac{1}{\eta_{\lambda}}  L (f\circ \Phi_{\lambda}) + \frac{1}{\sqrt{\eta_{\lambda}}}\sum_{\alpha=1}^m h_{\lambda, \alpha }(\Sigma_{\alpha}(f)\circ \Phi_{\lambda})= \mu(f) \circ \Phi_{\lambda} \]
		and this concludes the proof.
	\end{proof}

The following theorem establishes how to test the conditions under which it is possible to take the derivative inside the expectation in the proof of Theorem \ref{integrationbyparts1}. Let us introduce
\begin{equation}\label{eq:P} P (\lambda) = \frac{d \mathbb{Q}_{\lambda}}{d \mathbb{P}}
\left( \int_0^{f_{-\lambda }(t)} L(F) \circ
\Phi_{\lambda} (X_s, s) f^\prime_\lambda(s) d s \right). \end{equation}

\begin{theorem}\label{technicalderivability}
	Consider a SDE $(\mu, \sigma)$ with a symmetry $(Y, \tau, H)$ satisfying the Hypothesis \ref{Hp:A}. Introducing the one-parameter group $T_\lambda=(\Phi_\lambda,\eta_\lambda,h_\lambda)$ associated with the symmetry, one has that
\[
\mathbb{E}_{\mathbb{{P}}}\left[|\partial^2_{\lambda} P(\lambda)|\right]< \infty,
\] 	
where $P(\lambda)$ is given by the expression \eqref{eq:P}.
\end{theorem}
\begin{proof}
Let us derive $P(\lambda)$ with respect to $\lambda$  obtaining
\[ \partial_{\lambda} P (\lambda) = P (\lambda) \left[-m(t)+\int_0^T \partial_{\lambda} h_{\alpha, \lambda} (X_t, t) d
W^{\alpha}_t - \int_0^T h_{\alpha, \lambda} (X_t, t) \partial_{\lambda}
h_{\alpha, \lambda} (X_t, t) d t \right] \]
\[ + \frac{d \mathbb{Q}_{\lambda}}{d \mathbb{P}} \left( \int_0^{f_{-\lambda }(t)}  Y(L(F)) \circ
\Phi_{\lambda} (X_s, s) f^\prime_\lambda(s) d s \right)+ \frac{d \mathbb{Q}_{\lambda}}{d \mathbb{P}} \left( \int_0^{f_{-\lambda }(t)}  (L(F)) \circ \Phi_{\lambda} (X_s, s)  \partial_{\lambda}f^{\prime}_\lambda(s) d s \right)  \]
The second derivative becomes:
\begin{align*}
\partial_{\lambda}^2 P (\lambda) =& \frac{d \mathbb{Q}_{\lambda}}{d \mathbb{P}} \left( \int_0^{f_{-\lambda} (t)} L(F) \circ
\Phi_{\lambda} (X_s, s)f^{\prime}_\lambda(s) d s \right) \\
&\left[-m(t)+\int_0^T \partial_{\lambda} h_{\alpha, \lambda} (X_t, t) d
W^{\alpha}_t - \int_0^T h_{\alpha, \lambda} (X_t, t) \partial_{\lambda}
h_{\alpha, \lambda} (X_t, t) d t \right]^2 +\\
& +  \frac{d \mathbb{Q}_{\lambda}}{d \mathbb{P}} \left( \int_0^{f_{-\lambda}  (t)} Y(L(F)) \circ
 \Phi_{\lambda} (X_s, s)f^{\prime}_\lambda(s) d s \right)  \left[-m(t)+\int_0^T \partial_{\lambda} h_{\alpha, \lambda} (X_t, t) dW^{\alpha}_t \right.\\
 &\left. - \int_0^T h_{\alpha, \lambda} (X_t, t) \partial_{\lambda}
 h_{\alpha, \lambda} (X_t, t) d t \right]+ 2\frac{d \mathbb{Q}_{\lambda}}{d \mathbb{P}} \left( \int_0^{f_{-\lambda}  (t)} L(F) \circ
\Phi_{\lambda} (X_s, s) \partial_{\lambda}f^{\prime}_\lambda(s) d s \right)\times\\
&\times\left[-m(t)+\int_0^T \partial_{\lambda} h_{\alpha, \lambda} (X_t, t) d
W^{\alpha}_t - \int_0^T h_{\alpha, \lambda} (X_t, t) \partial_{\lambda}
h_{\alpha, \lambda} (X_t, t) d t \right]\\
&+ \frac{d \mathbb{Q}_{\lambda}}{d \mathbb{P}}\left( \int_0^{f_{-\lambda}  (t)} Y(L(F)) \circ
 \Phi_{\lambda} (X_s, s)f^{\prime}_\lambda(s) d s \right)\left[ \int_0^T \partial^2_{\lambda} h_{\alpha, \lambda} d
W^{\alpha}_t -\int_0^T  (\partial_{\lambda}
h_{\alpha, \lambda})^2 d t \right.\\
&\left. - \int_0^T h_{\alpha, \lambda} \partial^2_{\lambda}
h_{\alpha, \lambda} d t \right]+ 2\frac{d \mathbb{Q}_{\lambda}}{d \mathbb{P}} \left( \int_0^{f_{-\lambda} (t)} Y(Y(L(F))) \circ
\Phi_{\lambda} (X_s, s)f^{\prime}_\lambda(s) d s \right)\\
&+ \frac{d \mathbb{Q}_{\lambda}}{d \mathbb{P}} \left( \int_0^{f_{-\lambda}  (t)} Y(L(F)) \circ
\Phi_{\lambda} (X_s, s) \partial_{\lambda}f^{\prime}_\lambda(s) d s + \int_0^{f_{-\lambda}  (t)} L(F) \circ
\Phi_{\lambda} (X_s, s)\partial^2_{\lambda}f^{\prime}_\lambda(s) d s\right).
\end{align*}
To prove the theorem we have to verify  that each term of the previous sum is in $L^1$.
Starting from the first term of the second derivative of $P(\lambda) $ and using Radon-Nikodym theorem, we have that the expectation with respect to $\mathbb{P} $ is equal to
\[ \mathbb{E}_{\mathbb{Q}_\lambda}\left[m^2(t) \left( \int_0^{f_{-\lambda}  (t)} L(F) \circ
\Phi_{\lambda} (X_s, s)f^{\prime}_\lambda(s) d s \right)\right]-\mathbb{E}_{\mathbb{Q}_\lambda}\left[m(t )\left(\int_0^{\mathcal{T}} \frac{1}{\eta_{\lambda}}  H_{\alpha} (X_t, t) dW^{\alpha,\lambda}_{t_\lambda}\right) \right]\]
\[+ \mathbb{E}_{\mathbb{Q}_\lambda}\left[ \left( \int_0^{f_{-\lambda}  (t)} L(F) \circ \Phi_{\lambda} (X_s, s)f^{\prime}_\lambda(s) d s \right) \left(\int_0^{\mathcal{T}} \frac{1}{\eta_{\lambda}} H_{\alpha} (X_t, t)  dW^{\alpha,\lambda}_{t_\lambda}\right)^2\right],\]
where Proposition \ref{flow} and Theorem \ref{theo:theg} have been used.
Applying our time change to the stochastic integral (see, e.g., \cite{Oksendal}) we obtain
\[  \int_0^{\mathcal{T}} \frac{1}{\eta_{\lambda}} H_\alpha
\circ \Phi_{\lambda} (X_{t}, t)
 d W^{\alpha, \lambda}_{t_{\lambda}} =
 \int_0^{f_{- \lambda} (\mathcal{T})} \frac{1}{\eta_{\lambda}} H_\alpha
\circ \Phi_{\lambda} (X_{f_{- \lambda} (t_{\lambda})}, f_{- \lambda}
(t_{\lambda})) d W^{\alpha, \lambda}_{t}, \]
where
\[ t_{\lambda} {= f }_{\lambda} (t), \quad d W^{\alpha,
	\lambda}_{t_{\lambda}} = \sqrt{\eta_{\lambda}} (d W^{\alpha}_{f_{-
		\lambda} (t_{\lambda})} - h_{\alpha, \lambda} (X_{f_{- \lambda}
	(t_{\lambda})}, f_{- \lambda} (t_{\lambda})) d t_{\lambda}). \]
By Girsanov theorem, and by definition of general stochastic transformation,
we have that $W^{\alpha, \lambda}_{t_{\lambda}}$ is a $\mathbb{Q}_{\lambda}-$ Brownian motion.
On the other hand, by writing
\[ X_{t_{\lambda}}^{\lambda} = \Phi_{\lambda} (X_{f_{- \lambda}
	(t_{\lambda})}, f_{- \lambda} (t_{\lambda})) \]
and since $(Y, \tau, H)$ is a
symmetry of $(\mu, \sigma)$, for example for the second term, we obtain
\[\mathbb{E}_{\mathbb{Q}_{\lambda}}\left[ \left( \int_0^{{f_{- \lambda} (t)}} L(F) \circ \Phi_{\lambda} (X_s, s)(s) f^{\prime}_\lambda(s) d s \right)\int_0^{f_{- \lambda} (\mathcal{T})} \frac{1}{\eta_{\lambda}} H_\alpha(X^{\lambda}_{t_{\lambda}},t_{\lambda})
 d W^{\alpha, \lambda}_{t_{\lambda}} \right] \]
\[ =\mathbb{E}_{\mathbb{P}} \left[ \left( \int_0^{t} L(F) (X_s, s) d s \right)\int_0^{\mathcal{T}} H_\alpha
 (X_{ t}, t) d W^{\alpha}_{t} \right] \]
because the process $X_t$ under $\mathbb{P}$ has the same
distribution as the process $X^{\lambda}_{t_{\lambda}}$ under $\mathbb{Q}_{\lambda}$.
Now by Holder's inequality
\[
\left|\mathbb{E}_{\mathbb{P}} \left[ \left( \int_0^{\mathcal{T}} H_{\alpha} (X_t) d
W^{\alpha}_t \right) \left( \int_0^{t} L(F)(X_s, s) d s \right)\right]\right| \leq
\]
\[
\left(\mathbb{E}_{\mathbb{P}} \left[ \left( \int_0^{\mathcal{T}} H_{\alpha} (X_t) d
W^{\alpha}_t \right)^2 \right]\right)^{1/2}\left(\mathbb{E}_{\mathbb{P}} \left[ \left( \int_0^{t}  L(F) (X_s, s) d s  \right)^2 \right]\right)^{1/2}\leq
\]
\[
\leq||F||_{C^2} \left(\mathbb{E}_{\mathbb{P}} \left[  \int_0^{\mathcal{T}} | H_{\alpha} (X_t)|^2 dt  \right]\right)^{1/2} \left(\mathbb{E}_{\mathbb{P}} \left[ \int_0^{t}  |\mu (X_s, s)|^2 d s   \right]+\mathbb{E}_{\mathbb{P}}\left[ \int_0^{t}  |\sigma (X_s, s)|^2 d s   \right]\right)^{1/2} \]
where It\^o's isometry has been used.

Being $ H_{\alpha}, \mu, \sigma \in L^2 $  by Hypothesis A, we have proved that the first term of $\partial^2_{\lambda} P(\lambda) $ is in $L^1$.

Let us consider the  fourth term of $\partial_{\lambda}^2 P (\lambda)  $. By Proposition \ref{flow}
\[ \partial_{\lambda}^2 h_{\alpha, \lambda} = \partial_{\lambda} \left(
\frac{1}{\sqrt{\eta_{\lambda}}} H_{\alpha} \circ \Phi_{\lambda} \right) = -
\frac{\tau (f_{\lambda} (t))}{2 \sqrt{\eta_{\lambda}}} H_{\alpha}  \circ
\Phi_{\lambda} + \frac{1}{\sqrt{\eta_{\lambda}}} Y (H_{\alpha}) \circ
\Phi_{\lambda}, \]
and so we have
\[ \int_0^{\mathcal{T}} \partial_{\lambda}^2 h_{\alpha, \lambda} (X_t, t) d
W^{\alpha}_t - \int_0^{\mathcal{T}}  ((\partial_{\lambda} h_{\alpha, \lambda} (X_t,
t))^2 - h_{\alpha, \lambda} (X_t, t) \partial_{\lambda}^2 h_{\alpha,
	\lambda} (X_t, t)) d t = \]
\[ = \int_0^{\mathcal{T}} \frac{1}{\eta_{\lambda}} \left( Y (H_{\alpha})
(X^{\lambda}_{t_{\lambda}}) - \frac{\tau (t_\lambda)}{2} H_{\alpha}
(X^{\lambda}_{t_{\lambda}}) \right) d W^{\alpha,
	\lambda}_{t_{\lambda}} - \int_0^{\mathcal{T}} \frac{1}{\eta^2_\lambda}\left(  H_{\alpha}
(X^{\lambda}_{t_{\lambda}})\right)^2 d t_{\lambda}. \]
Applying our time change we get that the previous expression is equal to
\[ \int_0^{f_{-\lambda} (\mathcal{T})} \frac{1}{\eta_{\lambda}} \left( Y (H_{\alpha})
(X^{\lambda}_{t_{\lambda}}) - \frac{\tau (f_{\lambda} (t))}{2} H_{\alpha}
(X^{\lambda}_{t_{\lambda}}) \right) d W^{\alpha,
	\lambda}_{t} - \int_0^{f_{-} (\mathcal{T})}\frac{1}{\eta^2_\lambda}\left(  H_{\alpha}
(X^{\lambda}_{t_{\lambda}})\right)^2 d t. \]
The expectation with respect to $\mathbb{P}$ of the third term is equal to
\[
\mathbb{E}_{\mathbb{Q}_{\lambda}}  \left[
\left( \int_0^{f_{-\lambda (t)}} L(F) \circ
\Phi_{\lambda} (X_s, s)f^{\prime}_\lambda(s) d s \right)\left( \int_0^{f_{-\lambda} (\mathcal{T})} \frac{1}{\eta_{\lambda}} \left( Y (H_{\alpha})
(X^{\lambda}_{t_{\lambda}}) - \frac{\tau (f_{\lambda} (t))}{2} H_{\alpha}
(X^{\lambda}_{t_{\lambda}}) \right) d W^{\alpha,
	\lambda}_{t}\right)\right]
\]
\[ - \mathbb{E}_{\mathbb{Q}_{\lambda}}  \left[
\left( \int_0^{f_{-\lambda (t)}} L(F) \circ
\Phi_{\lambda} (X_s, s)f^{\prime}_\lambda(s) d s \right)\left(\int_0^{f_{-} (\mathcal{T})}\frac{1}{\eta^2_\lambda}\left(  H_{\alpha}
(X^{\lambda}_{t_{\lambda}})\right)^2 d t\right) \right].
\]

Since $(Y, \tau, H)$ is a
symmetry of $(\mu, \sigma)$ then the last expectation is equal to
\[
\mathbb{E}_{\mathbb{P}}  \left[
\left( \int_0^{t} L(F)(X_s, s) d s \right)\left( \int_0^{\mathcal{T}} \left( Y (H_{\alpha})
(X_{t}) - \frac{1}{2} H_{\alpha}
(X_{t}) \right) d W^{\alpha}_{t} - \int_0^{ \mathcal{T}}\left(  H_{\alpha}
(X_{t})\right)^2 d t.\right) \right]
\]
By using Holder inequality and It\^o's isometry we can finish the proof.
All the other terms can be handled in a similar way.
\end{proof}

\section{Examples}

\subsection{Brownian motion}

Let us consider as first example the one-dimensional Brownian motion
\[
\begin{pmatrix}
dX_t \\
dZ_t
\end{pmatrix}
=
\begin{pmatrix}
d{W_t} \\
dt
\end{pmatrix}
\]
If we look for quasi Doob symmetries of the form $V=(Y,\tau, H)$, with $Y= f(x,z)\partial_x+m(x,z)\partial_z$, $\tau=\tau(x,z)$ and $H=H(x,z)$, the determining equations are
\begin{align}
L(f)+H&=0 \notag \\
L(m)-\tau &=0 \notag \\
f_x-\frac 12 \tau &=0 \notag \\
m_x &=0 \notag
\end{align}
where $L= \frac 12 \partial_{xx}+\partial_z$ is the generator associated with the SDE. From the last equation we have that $m=\beta(z)$, and from the second one we get $\tau=\beta'(z)$. Therefore, the third equation gives $f=\frac 12 \beta'(z) x$ and from the first one we can compute $H=-L(f)=-\frac 12 x \beta''(z)$ and
we have the following  family of quasi Doob symmetries for the Brownian motion
\[
V=\Biggr(
\begin{pmatrix}
\frac{1}{2}\beta'(z)x \\
\beta(z)
\end{pmatrix}
,\beta'(z),-\frac{1}{2}\beta''(z)x\Biggl),
\]
where $\beta(z)$ is an arbitrary deterministic function of time.
In order to find the one-parameter group $ \Phi_\lambda (x,z)$ corresponding to the vector field $Y$ we solve the following system of differential equations
\begin{eqnarray*}
	\frac {dx}{d\lambda}&=&\frac 12 \beta'(z) x\\
	\frac {dz}{d\lambda}&=&\beta(z) \\
\end{eqnarray*}
and we find $z_{\lambda}= M(\lambda, z_0)$  and $x_\lambda =x_0 \exp (\frac 12 \int \beta'(z_{\lambda}) d\lambda )$.
Moreover, since the equation for the one-parameter group associated with $\eta$ is
\[
\frac {d\eta}{d\lambda}=\eta ( \tau\circ \Phi_{\lambda})=\eta \beta'(z_\lambda)
\]
we get $\eta_{\lambda}=\eta_0\exp (\int \beta'(z_{\lambda}) d\lambda )$. Finally, we have to solve the equation for $h_{\lambda}$, i. e.
\[
\frac {dh}{d\lambda}=\frac {1}{\sqrt{\eta_\lambda}}(H\circ\Phi_{\lambda})
=-\frac {x_0}{2\sqrt{\eta_0}}\beta''(z_{\lambda})
\]
and, choosing $\eta_0=1$ we get
\[
h_{\lambda}= -\frac {x_0}{2}\int\beta''(z_{\lambda}) d\lambda .
\]
Considering $(x,z)=(x_0,z_0)$ as the starting point and exploiting the expression for $h_{\lambda}$, we can find  $\mathfrak{h}_\lambda$: in fact, Proposition \ref{PropDoobTransf} ensures that
\[
\mathfrak{h}_{\lambda}= -\frac {x^2}{4}\int \beta''(z_{\lambda}) d\lambda.
\]
Finally, we can explicitly compute the expression of the function  $G_{\mathfrak{h}_{\lambda}}=\frac 12 h_{\lambda}^2+L(\mathfrak{h}_{\lambda})$  appearing in  Proposition \ref{PropDoobTransf}, i.e.
\[
G_{\mathfrak{h}_{\lambda}}=\frac 14 x^2 \left [ \int \beta''(z_{\lambda}) d\lambda\right]^2-\frac 14 x^2 \int  \beta'''(z_{\lambda})d \lambda -\frac 14\int \beta''(z_{\lambda}) d\lambda
\]
In order to check that conditions of Hypothesis A are satisfied we explicitly compute the following expressions
\[ H_{\alpha}=-\frac 12 \beta''(z) x \]
\[Y( H_{\alpha})=\frac 12x\left[ -\frac 12 \beta'(z)\beta''(z) -\beta(z)\beta'''(z)\right] \]
\[ L(Y^i)= \begin{pmatrix}
\frac 12 \beta''(z) x \\
\beta'(z)
\end{pmatrix} \]
\[ \Sigma_{\alpha}(Y^i)= \begin{pmatrix}
\frac 12 \beta'(z)  \\
0
\end{pmatrix} \]
\[ L(Y(Y^i))= \begin{pmatrix}
\frac 12 x\left[\frac 12 (\beta'(z))^2 +\beta(z)\beta''(z)\right] \\
(\beta'(z))^2+\beta(z)\beta''(z)
\end{pmatrix} \]
\[ \Sigma_{\alpha}(Y(Y^i))= \begin{pmatrix}
\frac 14 (\beta'(z))^2 +\frac 12\beta(z)\beta''(z) \\
0
\end{pmatrix} \]
Since the previous expressions are continuous in $z$ and at most with linear growth in $x$ we get that $ H_\alpha, Y, L(Y), \Sigma_{\alpha}, L(Y(Y^i)), \Sigma_{\alpha}(Y(Y^i)) $ are in $L^2(\Omega) $ since are polynomials of Gaussian r.v.

Applying Theorem \eqref{integrationbyparts1} we get our integration by parts formula (with $m=\beta(z)$)
\[
\mathbb{E}_{\mathbb{P}}\left[F^\prime(X_t)\frac{1}{2} \beta'(Z) X_t\right]=
\mathbb{E}_{\mathbb{P}}\left[F(X_t)\left(\int_{0}^{\mathcal{T}}(\frac{1}{2}\beta''(s) X_sdW_s\right)\right] + \beta(t)\mathbb{E}_{\mathbb{P}}\left[\frac{\partial_{xx}}{2}F(X_t)\right].
\]
Since we started by a quasi Doob symmetry, by writing the Doleans-Dade exponential according to Definition \ref{quasidoob} and considering Corollary \ref{corollary_integration} with $ k(x,z)=- \frac{x^2}{4}\beta''(z)   $ the previous integration by parts formula  admits also the following simplyfied form
\[
\mathbb{E}_{\mathbb{P}}\left[F^\prime(X_t)\frac{1}{2} \beta'(Z) X_t\right]=
\]
\[
\mathbb{E}_{\mathbb{P}}\left[F(X_t)\left(\frac{X^2_\mathcal{T}}{4}\beta''(Z)-\frac{X^2_0}{4}\beta''(Z)-\int_{0}^{\mathcal{T}}\left(\frac{X^2_s}{4}\beta'''(Z)+\frac 14\beta''(Z)\right)ds\right)\right] + \beta(t)\mathbb{E}_{\mathbb{P}}\left[\frac{\partial_{xx}F(X_t)}{2}\right].
\]

A standard application of It\^o formula to the function $g(X,Z)=\frac{X^2}{4}\beta''(Z)$ allows us to verify that the two  integration by parts formulas are consistent.

\subsection{Generalized Ornstein Uhlenbeck process }

We consider the generalized OU process
\[
\begin{pmatrix}
dX_t \\
dZ_t
\end{pmatrix}
=
\begin{pmatrix}
-v'(X_t) \\
1
\end{pmatrix}
dt+
\begin{pmatrix}
1 \\
0
\end{pmatrix}
dW_t,
\]
where $W_t$ is a one dimensional Brownian motion and $v'(x)$ is the derivative of a function $v(x) $, which is a polynomial with an even degree with positive leading coefficients.
If we look for quasi Doob symmetries of the form $V=(Y,\tau, H)$, with $Y= f(x,z)\partial_x+m(x,z)\partial_z$, $\tau=\tau(x,z)$ and $H=H(x,z)$, the determining equations are
\begin{align}
fv''(x)+L(f)+H+\tau v'(x)&=0 \notag \\
L(m)-\tau &=0 \notag \\
f_x-\frac 12 \tau &=0 \notag \\
m_x &=0 \notag
\end{align}
where $L= \frac 12 \partial_{xx}-v'(x)\partial_x +\partial_z$ is the generator associated with the SDE.
In particular, if we look for quasi Doob symmetries with $\tau=0$ we get $m(x,z)=const$ and $f(x,z)=\beta(z)$. Moreover, from the first equation we find $H(x,z)= -\beta(z) v''(x)-\beta'(z)$ and, choosing $m=0$
we have an infinite-dimensional family of infinitesimal symmetries of the form
\[
V=\Biggr(
\begin{pmatrix}
\beta(z) \\
0
\end{pmatrix}
,0,-\beta(z)v''(x)-\beta'(z)\Biggl)
\]
where $\beta(z)$ is an arbitrary deterministic function of z.

In order to find the one-parameter group $ \Phi_\lambda (x,z)$ corresponding to the vector field $Y$ we solve the following system of differential equations
\begin{eqnarray*}
	\frac {dx}{d\lambda}&=&\beta(z) \\
	\frac {dz}{d\lambda}&=&0 \\
\end{eqnarray*}
and we find $ \Phi_\lambda (x,z)=(\beta(z)\lambda+x, z)$.
Moreover, since the equation for the one-parameter group associated with $\eta$ is
\[
\frac {d\eta}{d\lambda}=\eta ( \tau\circ \Phi_{\lambda})=0
\]
we get $\eta_{\lambda}=\eta_0=1$. Therefore, solving the equation for $h_{\lambda}$, i.e.
\[
\frac {dh}{d\lambda}=\frac {1}{\sqrt{\eta_\lambda}}(H\circ\Phi_{\lambda})
=\left[ -\beta(z)v''(x+\lambda\beta(z))-\beta'(z)\right]
\]
we find
\[
h_{\lambda}= -v'(x+\lambda\beta(z))+v'(x)-\lambda\beta'(z).
\]
The following step consists in writing the explicit expression for $\mathfrak{h}$. Using Proposition \ref{PropDoobTransf} we have
\[
h_{j \lambda}(x)=\sigma_j^i(x)\partial_{x^i}(\mathfrak{h}_{\lambda})(x)
\]
that means
\[
\mathfrak{h}_{\lambda}=\int_{0}^{x}\left( -v'(s+\lambda\beta(z))+v'(s)-\lambda\beta'(z_0)\right) ds= -v(x+\lambda\beta(z))+ v(\lambda\beta(z))+v(x)-v(0)-x\lambda\beta'(z).
\]
Now, we can explicitly compute the expression of the function  $G_{\mathfrak{h}_{\lambda}}=\frac 12 h_{\lambda}^2+L(\mathfrak{h}_{\lambda})$  appearing in  Proposition \ref{PropDoobTransf}, i.e.
\[
G_{\mathfrak{h}_{\lambda}}=-\frac 12 v''(x+\lambda\beta(z))+\frac 12 v''(x)-\frac 12 v'(x)^2+\lambda \beta'(z) v'(\lambda\beta(z))-\lambda x\beta''(z)+\frac 12[v'(x+\lambda\beta(z))]^2 +\frac 12 \lambda^2\beta'(z)^2.
\]

In order to check that conditions of Hypothesis A are satisfied we explicitly compute
\[ H_{\alpha}=-\beta(z) v''(x) -\beta'(z) \]
\[Y( H_{\alpha})=-(\beta(z))^2v'''(x)  \]
\[ L(Y^i)= \begin{pmatrix}
 \beta'(z)  \\
0
\end{pmatrix} \]
\[ \Sigma_{\alpha}(Y^i)= \begin{pmatrix}
0 \\
0
\end{pmatrix} \]
\[ L(Y(Y^i))= \begin{pmatrix}
0 \\
0
\end{pmatrix} \]
\[ \Sigma_{\alpha}(Y(Y^i))= \begin{pmatrix}
0 \\
0
\end{pmatrix}. \]
Taking $\varphi (x) = \exp (K v (x))$ we have that
\[ \partial_t(\varphi)+ L (\varphi) = - K  (v' (x))^2 \varphi (t, x) +
\frac{1}{2} K v'' (x) \varphi (t, x) + \frac{1}{2} (K)^2  (v' (x))^2
\varphi (t, x). \]
Since $v$ by hypothesis is a polynomial with an even degree with positive leading coefficients, then there exist $K,M>0$ such that the following inequality holds
\[ \left( - K + \frac{1}{2} (K)^2 \right) (v' (x))^2 + \frac{1}{2} K v''
(x) \leq M . \]
Furthermore $\varphi = \exp (K v (x))$ is a Lyapunov function for OU process. Indeed, since $v'' (x)  $ and $v''' (x) $ are polynomials while $\varphi(x) $ grows more than exponentially, then
\[ | v'' (x) |^p, | v''' (x) |^p \leq c_1\varphi(x)+c_2. \]

Finally, we can apply Theorem \ref{integrationbyparts1}, obtaining the following explicit integration by parts formula
\[
\mathbb{E}_{\mathbb{P}}[\beta (t) F^\prime(X_t)]=
\mathbb{E}_{\mathbb{P}}[F(X_t)(\int_{0}^{\mathcal{T}}(\beta (s) v^{''} (X_s)  + \beta^\prime (s))dW_s)],
\]
where we used the assumption that $m(z)=0$.\\

Applying Corollary \ref{corollary_integration} with
\[
k(x,z)=\partial_{\lambda}\mathfrak{h}_{\lambda}|_{\lambda=0}=-v'(x)\beta(z)+v'(0)\beta(z)-x\beta'(z)
\]
we have
\[
\mathbb{E}_{\mathbb{P}}[\beta (t) F^\prime(X_t)]=
\mathbb{E}_{\mathbb{P}}\left[F(X_t)(v^\prime(X_{t})\beta(t)+X_{t}\beta^\prime(t)-v'(0)\beta(t)-(v^\prime(X_{0})\beta(0)+X_{0}\beta^\prime(0)-v'(0)\beta'(0))\phantom{-\int_{0}^{\mathcal{T}}} \right.
\]
\[
\left.-\int_{0}^{\mathcal{T}}(\frac 12 \beta (s) v^{'''} (X_s)  + X_s\beta^{''}(s)- v^\prime (0)\beta^\prime (s)- v^{'} (X_s)  v^{''} (X_s) \beta(s))ds\right]
\]
and by applying It\^o formula to the function $g(x,z)=v^\prime (x)\beta(z)+x\beta^\prime (z)-v'(0)\beta(z) $ it is possible to verify that the two formulas are consistent.

\subsection{Bessel process}

We consider the one-dimensional Bessel process
\[
\begin{pmatrix}
dX_t \\
dZ_t
\end{pmatrix}
=
\begin{pmatrix}
\frac {a}{X_t} \\
1
\end{pmatrix}
dt+
\begin{pmatrix}
1 \\
0
\end{pmatrix}
dW_t,
\]
where $W_t$ is a one dimensional Brownian motion and $a\geq \frac{5}{2}$ is a real constant.

The determining equations for quasi Doob symmetries of the form $V=(Y,\tau, H)$, with $Y= f(x,z)\partial_x+m(x,z)\partial_z$, $\tau=\tau(x,z)$ and $H=H(x,z)$ are
\begin{align}
\frac {a}{x^2}f+L(f)+H-\tau \frac ax &=0 \notag \\
L(m)-\tau &=0 \notag \\
f_x-\frac 12 \tau &=0 \notag \\
m_x &=0 \notag
\end{align}
where $L= \frac 12 \partial_{xx}+\frac ax \partial_x +\partial_z$ is the generator associated with the SDE.\\

It is easy to check that we have an infinite dimensional family of quasi Doob symmetries of the form
\[
V=\Biggr(
\begin{pmatrix}
\frac{1}{2}\beta'(z)x \\
\beta(z)
\end{pmatrix}
,\beta'(z),-\frac{1}{2}\beta''(z)x\Biggl)
\]
where $\beta(z)$ is an arbitrary deterministic function of time.

The one-parameter group $ \Phi_\lambda (x,z)$ corresponding to the vector field $Y$ has been already computed in Example 1 (Brownian Motion) as well as
\begin{align*}
\eta_{\lambda}	&=\eta_0\exp [\int \beta'(z_{\lambda}) d\lambda ]\notag \\
h_{\lambda} &=-\frac {x_0}{2}\int\beta''(z_{\lambda}) d\lambda \notag \\
\mathfrak{h}_{\lambda} &=-\frac {x^2}{4}\int \beta''(z_{\lambda}) d\lambda.\notag
\end{align*}
Therefore, we have only to compute the explicit expression of the function  $G_{\mathfrak{h}_{\lambda}}=\frac 12 h_{\lambda}^2+L(\mathfrak{h}_{\lambda})$  appearing in  Proposition \ref{PropDoobTransf}, i.e.
\[
G_{\mathfrak{h}_{\lambda}}=\frac 14 x^2 \left [ \frac 12\left(\int \beta''(z_{\lambda})\right)^2- \int \beta'''(z_{\lambda}) d\lambda \right]-\frac{1+2a}{4}  \int  \beta''(z_{\lambda})d \lambda.
\]

In order to check that conditions of Hypothesis A are satisfied we explicitly compute
\[ H_{\alpha}=-\frac 12 \beta''(z) x \]
\[Y( H_{\alpha})=\frac 12x\left[ -\frac 12 \beta'(z)\beta''(z) -\beta(z)\beta'''(z)\right] \]
\[ L(Y^i)= \begin{pmatrix}
\frac 12\left[ \frac ax\beta'(z)+ \beta''(z) x\right] \\
\beta'(z)
\end{pmatrix} \]
\[ \Sigma_{\alpha}(Y^i)= \begin{pmatrix}
\frac 12 \beta'(z)  \\
0
\end{pmatrix} \]
\[ L(Y(Y^i))= \begin{pmatrix}
\frac {a}{2x} \left[\frac 12 (\beta'(z))^2 +\beta(z)\beta''(z)\right]+\frac 12 x [2\beta'(z)\beta''(z)+\beta(z)\beta'''(z)] \\
(\beta'(z))^2+\beta(z)\beta''(z)
\end{pmatrix} \]
\[ \Sigma_{\alpha}(Y(Y^i))= \begin{pmatrix}
-\frac {a}{2x^2}\left[ \frac 12(\beta'(z))^2 +\beta(z)\beta''(z)\right]+\frac 12[2\beta'(z)\beta''(z)+\beta(z)\beta'''(z)] \\
0
\end{pmatrix} \]
If we take
\[ \varphi_1 (x, t) = \exp (g (t) x^2) \]
 we have
\[\partial_t(\varphi_1)+ L (\varphi_1) = g' (t) x^2 \varphi_1 + 2 a g (t)
\varphi_1 + \frac{1}{2} (2 g (t) \varphi_1 + 4 g (t)^2 x^2 \varphi_1). \]
To make the previous quantity less than $M\varphi_1 $ we require that
\[ g' (t) + 2 g (t)^2 \leq 0. \]
Let us analyze the behaviour in the neighborhood of zero by considering
\[ \varphi_2 (t, x) = x^{- \alpha} \quad \alpha > 0 .\] Since
\[ \partial_t(\varphi_2)+ L (\varphi_2) =  \alpha \left[-a  +
\frac{1}{2} (\alpha + 1)\right] x^{- \alpha - 2} \]
we have to ask
\[ \left( - a + \frac{1}{2} (\alpha + 1) \right) \leq 0 \Rightarrow
\alpha \leq 2 a - 1. \]
Requiring $\alpha > 4 $ we get
\[   a > \frac{5}{2} .\]
Since, for suitable constants $c_2, c_2$ and $c_3$, we have
\[
|H_\alpha|^2, |Y(H_\alpha) |^2, |L(Y^i) |^2, |\Sigma_{\alpha}(Y^i) |^2,|L(Y(Y^i)) |^2,|\Sigma_{\alpha}(Y(Y^i)) |^2\leq c_1\varphi_1+c_2\varphi_2 + c_3,
\]
the Hypothesis \ref{Hp:A} holds.
Applying Theorem \eqref{integrationbyparts1} we get
\[
\mathbb{E}_{\mathbb{{P}}}\left[\frac 12 \beta^\prime (t)X_t\partial_{x} F(X_t)\right]=m(t)\mathbb{E}_{\mathbb{{P}}}\left[\frac 12 \partial_{xx}F(X_t)+\frac{a}{X_t}\partial_{x}F(X_t)\right]+
\mathbb{E}_{\mathbb{{P}}}\left[F(X_t)\left(\int_{0}^{T}\frac 12 \beta^{''}(s)X_s  dW_s\right)\right].
\]
Since we have a quasi Doob symmetry, by Corollary  \ref{corollary_integration} with $k(x,z)=-\frac{x^2}{4}\beta''(z) $ we get
\[
\mathbb{E}_{\mathbb{{P}}}\left[\frac 12 \beta^\prime (t)X_t\partial_{x} F(X_t)\right]=m(t)\mathbb{E}_{\mathbb{{P}}}\left[\frac 12 \partial_{xx}F(X_t)+\frac{a}{X_t}\partial_{x}F(X_t)\right]
\]
\[
+\mathbb{E}_{\mathbb{{P}}}\left[F(X_t)
\left(\frac{X^2_t}{4}\beta''(t)-\frac{X^2_0}{4}\beta''(0)+\int_{0}^{T}(\frac 14 \beta^{''}(s)+\frac a2 \beta^{''}(s)+\frac{X^2_s}{4}\beta'''(s))ds \right)\right]
\]
and the  consistency of the two formulas can be established applying It\^o formula to the function $g(x,z)=\frac{x^2}{4}\beta''(z) $.

\subsection{Stochastic volatility models}

We consider the following family of volatility models, reprensenting a generalization of the well-known Heston model
\[
\begin{pmatrix}
d\nu_t\\
dS_t \\
dZ_t
\end{pmatrix}
=
\begin{pmatrix}
a\nu_t+b\\
\mu_0S_t \\
1
\end{pmatrix}
dt+
\begin{pmatrix}
\sigma_0\nu_t^{\alpha_1} & 0\\
0 & \nu_t^{\alpha_2} S_t\\
0&0
\end{pmatrix}\cdot
\begin{pmatrix}
dW^1_t\\
dW^2_t \\
\end{pmatrix}
\]
where $W_t=(W^1_t,W^2_t)$ is a two-dimensional Brownian motion and $a,b,\mu_0, \sigma_0$ are constants whose range of values will be discussed below.
For calculations it is convenient to introduce the transformation $X_t=\log(S_t)$, obtaining for the second equation
\[
dX_t=\left(\mu_0-\frac 12 \nu^{2\alpha_2}\right)dt + \nu_t^{\alpha_2}dW^2_t.
\]
The generator corresponding to the (complete) SDE is
\begin{equation}\label{generalized_Heston_generator}
L=\frac 12\sigma_0^2\nu^{2\alpha_1}\partial_{\nu\nu}+\frac 12 \nu^{2\alpha_2}\partial_{xx}+(a\nu+b)\partial_{\nu}+\left(\mu_0-\frac 12 \nu^{2\alpha_2}\right)\partial_x +\partial_z
\end{equation}
and the determining equation for Girsanov infinitesimal symmetries are
\begin{align}
af-L(f)-\sigma_0\nu^{\alpha_1}H_1+\tau (a\nu+b)&=0 \notag \\
-\alpha_2f\nu^{2\alpha_2-1}-L(g)-\nu^{\alpha_2}H_2+\tau(\mu_0-\frac 12 \nu^{2\alpha_2}) &=0 \notag \\
-L(m)+\tau &=0 \notag \\
\alpha_1\sigma_0 f \nu^{\alpha_1-1}-f_{\nu}\sigma_0\nu^{\alpha_1}+\frac 12 \tau\sigma_0\nu^{\alpha
	_1} &=0 \notag \\
-f_x\nu^{\alpha_2} &=0 \notag \\
-g_{\nu}\sigma_0\nu^{\alpha_1} &=0 \notag \\
\alpha_2 f \nu^{\alpha_2-1}-g_x\nu^{\alpha_2}+\frac 12 \tau \nu^{\alpha_2}  &=0 \notag \\
-m_{\nu}\sigma_0\nu^{\alpha_1}&=0 \notag \\
-m_x\nu^{\alpha_2}&=0
\end{align}
Solving these equations we find the following infinite dimensional family of Girsanov symmetries:
\[
V=\left(
\begin{pmatrix}
\frac{\beta'(z)}{2(1-\alpha_1)}\nu \\
\frac{(\alpha_2+1-\alpha_1)\beta'(z)}{2(1-\alpha_1)}x \\
\beta(z)
\end{pmatrix}
,\beta'(z),
\begin{pmatrix}H_1\\H_2\\ \end{pmatrix}
\right)
\]
where $\beta(z)$ is an arbitrary  function of time and
\[
H_1=\frac {1}{\sigma_0}\left[ \nu^{-\alpha_1}\left( \frac {b\beta'(1-2\alpha_1)}{2(1-\alpha_1)}\right)+\nu^{1-\alpha_1}\left(a\beta'-\frac {\beta''}{2(1-\alpha_1)}\right)\right]
\]
\[
H_2=\nu^{\alpha_2}\left[ \frac {\beta'(-\alpha_2+\alpha_1-1)}{4(1-\alpha_1)}\right] + \nu^{-\alpha_2}\left[ \frac{\mu_0\beta'(1-\alpha_1-\alpha_2)-x\beta''(\alpha_2+1-\alpha_1)}{2(1-\alpha_1)}\right].
\]
In order to find the one-parameter group $ \Phi_\lambda (\nu,x,z)$ corresponding to the vector field $Y$ we have to solve the following system of differential equations
\begin{eqnarray*}
	\frac {d\nu}{d\lambda}&=&\frac{\beta'(z)}{2(1-\alpha_1)} \nu\\
	\frac {dx}{d\lambda}&=&\frac {\beta'(z)(\alpha_2+1-\alpha_1)}{2(1-\alpha_1)}x \\
	\frac {dz}{d\lambda}&=&\beta(z) \\
\end{eqnarray*}
and we find
\begin{eqnarray*}
	\nu_{\lambda}&=&\nu_0\exp\left[\frac {M}{2(1-\alpha_1)}\right]\\
	x_{\lambda}&=&x_0\exp \left[\frac {M(\alpha_2+1-\alpha_1)}{2(1-\alpha_1)}\right]\\
	z_{\lambda}&=& N(\lambda, z_0)\\
\end{eqnarray*}
where $M=\int \beta'(z_{\lambda}) d\lambda$.
Moreover, since the equation for the one-parameter group associated with $\eta$ is
\[
\frac {d\eta}{d\lambda}=\eta ( \tau\circ \Phi_{\lambda})=\eta \beta'(z_\lambda)
\]
we get $\eta_{\lambda}=\eta_0\exp [M])$ and we can use the previous expression in order to find $h_{\lambda}$, recalling that $ \frac {dh_i}{d\lambda}=\frac {1}{\sqrt{\eta_{\lambda}}}H_i(\nu_\lambda,x_\lambda,z_\lambda)$.
Indeed, with long but straightforward computations we get
\begin{eqnarray*}
	h_{1\lambda}&=&-\frac {b(1-2\alpha_1)}{\sigma_0}\nu_0^{-\alpha_1}\left( \exp \left[\frac {-M}{2(1-\alpha_1)}\right]-1\right)+\frac {\nu^{1-\alpha_1}}{\sigma_0}\left( aM -\frac {\int \beta'' (z_{\lambda})d\lambda}{2(1-\alpha_1)}\right)\\
	h_{2\lambda}&=&\frac {(-\alpha_2+\alpha_1-1)}{2(\alpha_1+\alpha_2-1)}\nu_0^{\alpha_2}\left( \exp \left[\frac {M(\alpha_1+\alpha_2-1)}{2(1-\alpha_1)}\right]-1\right)
	+ \mu_0\frac {(1-\alpha_1-\alpha_2)}{(\alpha_1-\alpha_2-1)}\nu_0^{-\alpha_2}\times\\
	& &\times \left( \exp \left[\frac {M(\alpha_1-\alpha_2-1)}{2(1-\alpha_1)}\right]-1\right)
	 - x_0\frac {(1-\alpha_1+\alpha_2)}{2(1-\alpha_1)}\nu_0^{-\alpha_2}\int \beta''(z_{\lambda})d\lambda.
\end{eqnarray*}

In order to check the integrability condition, we remark that in this case we have a Girsanov symmetry that is not of quasi Doob type. On the other hand, since the first equation is independent of the second one, if we consider only the first and the last components of the vector field $Y$ we get a quasi Doob symmetry in the coordinates $(\nu, z)$ and we can look for the functions  $\mathfrak{h}_{1\lambda} $ and $G_{\mathfrak{h}_{1\lambda}}$ as in the previous examples. After long but straightforward computations we get
\[
\mathfrak{h}_{1\lambda} =\frac {1}{\sigma_0}\left( -b\nu^{-2\alpha_1+1}\left[ \exp \left( \frac{-M}{2(1-\alpha_1)}\right)-1\right]+\frac {\nu^{2-2\alpha_1}}{2(1-\alpha_1)}\left[aM-\frac {\int \beta''(z_{\lambda})d\lambda}{2(1-\alpha_1)}\right]\right)
\]
and
\begin{eqnarray*}
	G_{\mathfrak{h}_{1\lambda}}&=&\nu^{-1}\left[ b\alpha_1(1-2\alpha_1)\Omega_1\right]+\frac {(1-2\alpha_1)b^2\nu^{-2\alpha_1}}{2\sigma_0^2}\left[ -2\Omega_1+(1-2\alpha_1)\Omega_1^2\right]+\\
	&+&\frac {b\nu^{1-2\alpha_1}}{\sigma_0^2}\left[ -a(1-2\alpha_1)\Omega_1+\Omega_2+\frac {(\Omega_1+1)M'}{2(1-\alpha_1)}-(1-2\alpha_1)\Omega_1\Omega_2\right]+\\
	&+&\frac {\nu^{2-2\alpha_1}}{\sigma_0^2}\left[ \frac {aM'}{2(1-\alpha_1)}+a\Omega_2-\frac {\int \beta'''(z_{\lambda})d\lambda}{4(1-\alpha_1)^2}+\frac {\Omega_2^2}{2}\right]+ \frac {1-2\alpha_1}{2}\Omega_2
\end{eqnarray*}
where
\begin{eqnarray*}
	\Omega_1&=&\exp \left( \frac {-M}{2(1-\alpha_1)}\right)-1\\
	\Omega_2&=&aM-\frac {\int \beta''(z_{\lambda})d\lambda}{2(1-\alpha_1)}.
\end{eqnarray*}
In order to check that conditions of Hypothesis A are satisfied we explicitly compute
\[ H_{1}=\frac {1}{\sigma_0}\left[\nu^{-\alpha_1}\left(\frac {b\beta'(z)(1-2\alpha_1)}{2(1-\alpha_1)}\right)+ \nu^{1-\alpha_1}\left(a\beta'(z)-\frac {\beta''(z)}{2(1-\alpha_1)}\right)\right] \]
\[ H_{2}=\nu^{\alpha_2}\left[\frac {\beta'(z)(\alpha_1-\alpha_2-1)}{4(1-\alpha_1)}\right]+ \nu^{-\alpha_2}\left[\frac {\mu_0\beta'(z)(1-\alpha_1-\alpha_2)-x\beta''(z)(\alpha_2+1-\alpha_1)}{2(1-\alpha_1)}\right] \]
\[ Y(H_1)=\nu^{-\alpha_1} K_1(z) + \nu^{1-\alpha_1} K_2(z)\]
\[ Y(H_2)=\nu^{\alpha_2} K_3(z) + \nu^{-\alpha_2} K_4(z) + x\nu^{-\alpha_2}K_5(z)  \]
\[ L(Y^1)= \nu K_6(z)  + K_7(z)  \]
\[ L(Y^2)= \nu^{2\alpha_2} K_8(z)  + x K_9(z)  + K_{10}(z) \]
\[ L(Y^3)=\beta'(z) \]
\[ \Sigma_1(Y^1)= \nu^{\alpha_1} K_{11}(z) \]
\[ \Sigma_1(Y^2)= \Sigma_1(Y^3)=0 \]
\[ \Sigma_2(Y^1)= \Sigma_2(Y^3)=0 \]
\[ \Sigma_2(Y^2)= \nu^{\alpha_2}K_{12}(z) \]
where $K_7, K_8,K_{10},K_{11}$ and $K_{12}$  are some  continuous functions depending on  $\beta(z),\beta'(z)$,   $K_1, K_2, K_3, K_4, K_6$ and $K_9$ are some continuous functions depending on $\beta(z),\beta'(z), \beta''(z)$, and finally $K_5$ is a continuous function depending on  $\beta(z),\beta'(z), \beta''(z), \beta'''(z).$
To find the integration by part formula for the stochastic volatility model under study in this section, we need to prove the following  technical result.

\begin{lemma}\label{generalized_Heston}
Suppose that
\[ \alpha_1 > \frac{1}{2}, \alpha_2\leq \frac{1}{2}, a<0, b>0.\]
Then
\[
\phi_1(x,\nu)=\exp(k\nu) ,\quad k << 1
\]
\[
\phi_2(x,\nu) =\exp(k\nu^{-k^\prime}), \quad k,k^\prime << 1
\]
\[
\phi_3(x,\nu)=x^2(\nu^{-1}+c)+c_1\exp(k\nu)+c,
\]
are Lyapunov functions for the original model equation.
If $\alpha_1 =\alpha_2=\frac{1}{2} $	and $b>\sigma^2,$ then
\[
\phi_1(x,\nu)=\nu^{-1}
\]
\[
\phi_2(x,\nu) =\exp(k\nu),\quad k << 1
\]
\[
\phi_3(x,\nu)=x^2(\nu^{-1}+c)+c_1\exp(k\nu)+c_2\nu^{-1} + c_3
\]
are Lyapunov functions for the original model equation.
\end{lemma}
\begin{proof}
	We give the proof only for the case $\alpha_1=\alpha_2=\frac{1}{2}.$ Let us consider  $\phi_1(\nu)=\nu^{-1}$. If $L$ denote the infinitesimal generator given in \eqref{generalized_Heston_generator} we get
	\[
	\partial_t(\phi_1)+L(\phi_1)=\sigma_0^2 \cdot \nu^{-2}-(a\nu+b)\nu^{-2}=(\sigma_0^2-b)\nu^{-2}-a\nu^{-1}<M_1
	\]
	for some constant $M_1$ since by hypothesis $b>\sigma_0^2. $
	Taking $\phi_2=\exp(k\nu) $ we have
	\[
	\partial_t(\phi_2)+L(\phi_2)=\frac{1}{2}\sigma_0^2 k^2\nu \exp(k\nu)+(a\nu+b)k\exp(k\nu) =k[(\frac{1}{2}\sigma_0^2k+a)\nu+b]\exp(k\nu)<M_2
	\]
	for some constant $M_2 $ whenever $k<\frac{-2a}{\sigma_0^2} $, being $a<0$. Finally with $\phi_3(\nu)=x^2(\nu^{-1})+cx^2 + \psi(\nu)$, where $\psi(\nu)=c_1\exp(k\nu)+c_2\nu^{-1}+c_3,$ we obtain
	\[
	\partial_t(\phi_3)+L(\phi_3)=x^2L(\phi_1)+L(\psi(\nu))+1+\nu c+2(\mu_0-\frac{1}{2}\nu)x \nu^{-1}+2(\mu_0-\frac{1}{2}\nu)xc.
	\]
By the first part of the proof we have that $L(\phi_1)<M_1$ and $L(\psi(\nu))<M_3$ (for some $M_3>0$). This implies that
\[ x^2L(\phi_1)+L(\psi(\nu))<M_1 x^2+M_3< d_1 \phi_3 +e_1\]
for some $d_1,e_1>0$. Furthermore
\[ c \nu< \frac{c}{2}+\frac{\nu^2}{2} < d_2 \psi(\nu) + e_2< d_2 \phi_3 + e_2 .\]
In a similar way we get
\[2\left(\mu_0-\frac{1}{2}\nu\right)x \nu^{-1} = \mu_0 x^2+\frac{\mu_0}{\nu^2} +\frac{1}{2} +\frac{x^2}{2} < d_3 \phi_3 +e_3, \]
and
\[2\left(\mu_0-\frac{1}{2}\nu\right)xc <c\mu_0+c\mu_0 x^2+ \frac{c\nu^2}{2}+\frac{x^2}{2}< d_4 \phi_3 +e_4.\]
This means that there are some constants $c_4,c_5>0$ for which
	\[
	\partial_t(\phi_3)+L(\phi_3)\leq c_4\phi_3 + c_5,
	\]
	and, hence, we get the thesis.
\end{proof}

Finally, Theorem \eqref{integrationbyparts1} permits us to obtain the following integration by parts formula for our two-dimensional model:

\[
\mathbb{E}_{\mathbb{{P}}}\left[\frac{\beta'(z)}{2(1-\alpha_1)}\nu_t \partial_{\nu}F(\nu_t,X_t)+\frac{(\alpha_2+1-\alpha_1)\beta'(z)}{2(1-\alpha_1)}X_t \partial_{x} F(\nu_t,X_t)\right]=
\]
\[
m(t)\mathbb{E}_{\mathbb{{P}}}\left[\frac 12\sigma_0^2\nu^{2\alpha_1}\partial_{\nu\nu}F(\nu_t,X_t)+\frac 12 \nu^{2\alpha_2}\partial_{xx}F(\nu_t,X_t)+(a\nu+b)\partial_{\nu}F(\nu_t,X_t)+(\mu_0-\frac 12 v^{2\alpha_2})\partial_x F(\nu_t,X_t)\right]
\]
\[
-\mathbb{E}_{\mathbb{{P}}}\left[\left(\int_{0}^{\mathcal{T}}\frac {1}{\sigma_0}[ \nu^{-\alpha_1}\left( \frac {b\beta'(1-2\alpha_1)}{2(1-\alpha_1)}\right)+\nu^{1-\alpha_1}\left(a\beta'-\frac {\beta''}{2(1-\alpha_1)}\right)dW^1_s\right) F(X_t)\right]
\]
\[
-\mathbb{E}_{\mathbb{{P}}}\left[\left(\int_{0}^{\mathcal{T}}\nu^{\alpha_2}\left[ \frac {\beta'(-\alpha_2+\alpha_1-1)}{4(1-\alpha_1)}\right] + \nu^{-\alpha_2}[ \frac{\mu_0\beta'(1-\alpha_1-\alpha_2)-x\beta''(\alpha_2+1-\alpha_1)}{2(1-\alpha_1)}dW^2_s\right) F(X_t)\right]
\]
Applying Corollary \ref{corollary_integration} we obtain
\[
\mathbb{E}_{\mathbb{{P}}}\left[\frac{\beta'(z)}{2(1-\alpha_1)}\nu_t \partial_{\nu}F(\nu_t,X_t)+\frac{(\alpha_2+1-\alpha_1)\beta'(z)}{2(1-\alpha_1)}X_t \partial_{x} F(\nu_t,X_t)\right]=
\]
\[
m(t)\mathbb{E}_{\mathbb{{P}}}\left[\frac 12\sigma_0^2\nu^{2\alpha_1}\partial_{\nu\nu}F(\nu_t,X_t)+\frac 12 \nu^{2\alpha_2}\partial_{xx}F(\nu_t,X_t)+(a\nu+b)\partial_{\nu}F(\nu_t,X_t)+(\mu_0-\frac 12 v^{2\alpha_2})\partial_x F(\nu_t,X_t)\right]
\]
\[
-\mathbb{E}_{\mathbb{{P}}}\left[\left(\int_{0}^{\mathcal{T}}\frac {1}{\sigma_0}[ \nu^{-\alpha_1}\left( \frac {b\beta'(1-2\alpha_1)}{2(1-\alpha_1)}\right)+\nu^{1-\alpha_1}\left(a\beta'-\frac {\beta''}{2(1-\alpha_1)}\right)dW^1_s\right) F(X_t)\right]
\]
\[
-\mathbb{E}_{\mathbb{{P}}}\left[\left(\int_{0}^{\mathcal{T}}\nu^{\alpha_2}\left[ \frac {\beta'(-\alpha_2+\alpha_1-1)}{4(1-\alpha_1)}\right] + \nu^{-\alpha_2}[ \frac{\mu_0\beta'(1-\alpha_1-\alpha_2)-x\beta''(\alpha_2+1-\alpha_1)}{2(1-\alpha_1)}dW^2_s\right) F(X_t)\right].
\]

\section{Appendix A}

In this section we provide an integration by parts formula using a Lie's simmetries approach in a finite dimensional setting by taking smooth cylindrical functionals of the diffusion process.

Let us start by some useful definitions and by expliciting our main assumptions.

\begin{definition}\label{cylindrical_function}
	Let us introduce $\mathcal{S}=(t_1,t_2,...,t_k)$ with $t_1,t_2,...,t_k $ some fixed times such that $t_1>t_2>...>t_k$. A function $F^{\mathcal{S}}: M \rightarrow \Real$ is called a smooth cylindrical function if there exists a $\mathcal{C}^\infty $ function $f: M^k \rightarrow \Real  $ such that
	\[
	F^{\mathcal{S}}(X)=f(X_{t_1},X_{t_2},...,X_{t_k}).
	\]
\end{definition}

\begin{lemma} Let $F^{\mathcal{S}}(X) $ be a smooth cylindrical function according to Definition \ref{cylinder_function}.
	Introducing the process
	\[
	\tilde{F}_t:=\tilde{F}^{\mathcal{S}}(X_t):=f(X_{t_1\wedge t},X_{t_2\wedge t},...,X_{t_k\wedge t}),
	\]
	we have
	\begin{equation}
	\tilde{F}_t-\tilde{F}_0=\int_{0}^{t}\sum_{i=1}^{n}\partial^{s,\mathcal{S}}_i \tilde{F}\cdot \sum_{\alpha=1}^m\sigma_\alpha^idW^\alpha_s+\int_{0}^{t} L^{s,\mathcal{S}}(\tilde{F})ds
	\end{equation}
	where
	\[
	\partial^{s,\mathcal{S}}_i \tilde{F}=\sum_{j \le N_s^{\mathcal{S}}}\partial_{y^i_j}f
	\]
	with
	\[
	N_s^{\mathcal{S}}=\min\{ i, t_i>s,t_i \in {\mathcal{S}}\}
	\]
	and
	\[
	L^{s,\mathcal{S}}(F)=\frac{1}{2}\sum_{j_1,j_2 \le N_s({\mathcal{S}})}\sum_{i,l=1}^{n}\sum_{\alpha=1}^{m}\sigma_\alpha^i\sigma_\alpha^l\partial_{y^i_{j_1},y^l_{j_2}}f+\sum_{j \le N_s({\mathcal{S}})}\sum_{i=1}^{n}\mu^i(X_s)\partial_{y^i_j}f
	\]
\end{lemma}
\begin{proof}
	By applying multidimensional It\^o formula to the function $\tilde{F}^{\mathcal{S}}(X_t): M \rightarrow \Real $ we get
	\[
	\tilde{F}^{\mathcal{S}}(X_t)-\tilde{F}^{\mathcal{S}}(X_0)=\int_{0}^{t}\sum_{i=1}^n\sum_{\alpha=1}^m \partial_{i}\tilde{F}^{\mathcal{S}}(X_s)\sigma^i_\alpha dW^\alpha_s+
	\int_{0}^{t}L_s^{\mathcal{S}}(\tilde{F})(X_s)ds.\]
	We remark that $\tilde{F}^{\mathcal{S}}(X_s)$ is not a function of the $n$-dimensional v.a. $X_s$ but of a variable number of copies of $X_s$ depending on where the time $s$ falls. For example, if $t_k <s <t_{k-1}$, then $\tilde{F}^{\mathcal{S}}(X_s)=f(X_s,X_s,.,.,.,X_s,X_{t_k}) $, that is $\tilde{F}^{\mathcal{S}}(X_s) $ is a function of $k-1$ copies of $X_s$.
	
	Therefore, in the general case we recognize the following identities
	\[
	\partial_{i}\tilde{F}^{\mathcal{S}}(X_s)=\partial_{i}f(X_{t_1\wedge s},X_{t_2\wedge s},...,X_{t_k\wedge s})=\sum_{j \le N_s^{\mathcal{S}}}\partial_{y^i_j}f
	\]
	and
	\[
	\partial_{i,l}F^{\mathcal{S}}(X_s)=\partial_{i,l}f((X_{t_1\wedge s},X_{t_2\wedge s},...,X_{t_k\wedge s})=\sum_{j_1,j_2 \le N_s^{\mathcal{S}}}\partial_{y^i_{j_1},y^l_{j_2}}f
	\]
	where $y^i_j $ denotes the (effective) $j$ component of the derivative of the function $f$ with respect to $i$.
\end{proof}

We state our generalized integration by parts formula.

\begin{theorem}[Integration by parts formula]\label{integrationbyparts2}
	Let $(X.W)$ be a solution to the SDE \eqref{SDE1} and let $(Y, \tau ,H)$ be an infinitesimal stochastic symmetry for the SDE. Let us consider the one-parameter group  $T_\lambda=(\Phi_{\lambda}, \eta_\lambda, h_\lambda) $ associated with the infinitesimal symmetry according to Proposition \ref{flow}. 
	Taking a cylindrical function $F^{\mathcal{S}}$, with $\mathcal{S}=(t_1,t_2,...,t_k)$, according to Definition \ref{cylinder_function}, and assuming Hypothesis A,  the following integration by parts formula holds (with $t_{k+1}=0 $)
	\begin{multline}
	\sum_{i=1}^{k}\mathbb{E}_{\mathbb{{P}}}[-m(t_i)L^{t_i,\mathcal{S}}(\tilde{F})+m(t_{i+1})L^{t_{i+1},\mathcal{S}}(\tilde{F})]+\mathbb{E}_{\mathbb{{P}}} \left[
	\tilde{F} (X_t) \left( \int_0^t H_\alpha (X_s) dW^\alpha_s \right) \right]+\\
	+\mathbb{E}_{\mathbb{{P}}}[(Y(\tilde{F})^{\mathcal{S}})(X_t)]-\mathbb{E}_{\mathbb{{P}}}[(Y(\tilde{F})^{\mathcal{S}})(X_0)]=0
	\end{multline}
\end{theorem}

\begin{proof}
	Since the  random transformation $T_\lambda$ is a symmetry for the SDE, denoting by $X^\lambda_t=P_{T_\lambda}(X,W)$,  by definition of symmetry we have
	\[
	\mathbb{E}_{\mathbb{{P}}}[\tilde{F}^{\mathcal{S}}(X_t)]=\mathbb{E}_{\mathbb{{Q}_\lambda}}[\tilde{F}^{\mathcal{S}}(X^\lambda_t)].
	\]
	Using It\^o formula for a composite function $\tilde{F}(\Phi_{\lambda})$ ( Lemma \ref{cylinder_function}) we get
	\begin{equation}
	F(\Phi_{\lambda}(X_{t})=\int_{0}^{t}L^{s,\mathcal{S}}(F)( \Phi_{\lambda})^{f_{-\lambda}(\mathcal{S})}ds+\int_{0}^{t}\partial^{s,\mathcal{S}}_i (F)( \Phi_{\lambda})\sigma_\alpha^i dW^\alpha,
	\end{equation}
	which gives
	\[
	\mathbb{E}_{\mathbb{{P}}}[\tilde{F}^{\mathcal{S}}(X_t)]=\mathbb{E}_{\mathbb{{Q}_\lambda}}\left[\int_{0}^{t}L^{s,\mathcal{S}}(F)( \Phi_{\lambda})^{f_{-\lambda}(\mathcal{S})}ds_\lambda\right].
	\]
	Denoting by $f_{-\lambda}(\mathcal{S})=(f_{-\lambda}(t_1),...,f_{-\lambda}(t_k)) $ the inverse of the deterministic time change $f_\lambda (t)(\mathcal{S})=(f_{\lambda}(t_1),...,f_{\lambda}(t_k)) $ and applying a change of variables  we get, $\forall t \in [0,\mathcal{T}] $,
	\[
	\mathbb{E}_{\mathbb{{Q}_\lambda}}\left[\int_{0}^{t}L^{s,\mathcal{S}}(F)( \Phi_{\lambda})ds_\lambda\right]=\mathbb{E}_{\mathbb{{{Q}_\lambda}}}\left[\int_{0}^{f_{-\lambda}(t)} L^{s,\mathcal{S}}({F})( \Phi_{\lambda})(X_s)f^\prime_{\lambda}(\mathcal{S})ds\right]
	\]
	and performing a Radom-Nikodym measure change in the expectation on the right-hand side we obtain
	\[
	\mathbb{E}_{\mathbb{{P}}}\left[\tilde{F}^{\mathcal{S}}(X_t)\right]=\mathbb{E}_{\mathbb{{P}}}\left[\left.\frac{d\mathbb{Q}_{\lambda}}{d\mathbb{P}}\right|_{\mathcal{F}_t}\int_{0}^{f_{-\lambda}(t)}L^{s,\mathcal{S}}(\tilde{F})( \Phi_{\lambda})(X_s)f^\prime_{\lambda}(\mathcal{S})ds\right].
	\]

	Applying Fubini's theorem and putting $t_{k+1}=0$ we get
	\begin{equation}\label{invariance_principle5}
	\mathbb{E}_{\mathbb{{P}}}[\tilde{F}^{\mathcal{S}}(X_t)]=\mathbb{E}_{\mathbb{{P}}}\left[\left.\frac{d\mathbb{Q}_{\lambda}}{d\mathbb{P}}\right|_{\mathcal{F}_s}\sum_{i=1}^{k}\int_{f_{-\lambda}(t_i+1)}^{f_{-\lambda}(t_i)}L^{s,\mathcal{S}}(\tilde{F}) (\Phi_{\lambda})(X_s))f^\prime_{\lambda}(\mathcal{S})\right]ds.
	\end{equation}
	We want to take the derivative with respect to the parameter $\lambda$ of both sides in the previous equality. We note that in the left-hand side there is no dependence on the parameter $\lambda $. In both terms of the right-hand side we have to take the derivative inside the expectation. Denoting by $P(\lambda) = Z_\lambda\int_{0}^{f_{-\lambda}(t)}L^{s,\mathcal{T}}(F\circ \Phi_{\lambda})(X_s)f^\prime_{\lambda}(\mathcal{S})ds$  where $Z_\lambda=\frac{d\mathbb{Q}_{\lambda}}{d\mathbb{P}}$ , under Hypothesis \ref{Hp:A} and by  Theorem \ref{derivability_conditions}, we have that $\partial_{\lambda}^2 P(\lambda)\in L^2 $ and thus by Lemma \ref{lemma_derivative2}
	\[
	\partial_{\lambda}\mathbb{E}_{\mathbb{{P}}}[\tilde{g}(\lambda, X)]=\mathbb{E}_{\mathbb{{P}}}[\partial_{\lambda}\tilde{g}(\lambda, X)].
	\]
	
	Taking the derivative with respect to $\lambda $ in \eqref{invariance_principle5} and evaluating the result in $\lambda=0 $ we obtain
	\[
	\sum_{i=1}^{k}\mathbb{E}_{\mathbb{{P}}}[-m(t_i)L^{t_i,\mathcal{S}}(\tilde{F})+m(t_{i+1})L^{t_{i+1},\mathcal{S}}(\tilde{F})]+\mathbb{E}_{\mathbb{{P}}}\left[\left(\int_{0}^{t_1}H_\alpha(X_s)dW^\alpha_s\right)\left(\int_{0}^{t}L^{s,\mathcal{S}}(\tilde{F})(X_s)ds\right)\right]+
	\]
	\begin{equation}\label{derivativeslambda}
	+\mathbb{E}_{\mathbb{{P}}}\left[\int_{0}^{t_1}Y(L^{s,\mathcal{S}}(\tilde{F}))(X_s)ds\right]
	\end{equation}
	where the first term has been obtained by using the fundamental theorem of calculus for deriving both the $\lambda $ dependent extremes of integration, the second one by taking the partial derivative of the Doleans-Dade exponential process as in Definition \ref{defi:dea} and using the fact that, by Proposition \ref{flow}, $\partial_{\lambda}h_{\lambda, j}(x)|_{\lambda=0}=H_j(x)$ with the conditions $ \eta_0=1, \Phi_0 $,  the third one  by deriving the spatial flow $\Phi_\lambda$ according with Proposition \ref{flow} together with the property that $\Phi_0$ is the identity function and the fourth one by deriving with respect to $\lambda $ the term $f^\prime_{\lambda}(\mathcal{S})$ (as in Theorem \ref{integrationbyparts1}).
	Since $[Y,L]=-\tau L + H_\alpha \Sigma_{\alpha}$ we can write
	\[
	\mathbb{E}_{\mathbb{{P}}}\left[\int_{0}^{t_1}Y(L^{s,\mathcal{S}}(\tilde{F}))(X_s)ds\right]=\mathbb{E}_{\mathbb{{P}}}\left[\int_{0}^{t}L^{s,\mathcal{S}}(Y(\tilde{F})(X_s)ds\right]+
	\]
	\[
	- \mathbb{E}_{\mathbb{{P}}}\left[\int_{0}^{t_1} \tau(s)L^{s,\mathcal{S}}(\tilde{F})(X_s)ds\right]+
	\mathbb{E}_{\mathbb{{P}}}\left[\int_{0}^{t_1} H_\alpha (X_s) \sigma^i_{\alpha}\partial_i (\tilde{F})(X_s)ds\right]
	\]
	
	By It\^o formula we get
	\[
	\mathbb{E}_{\mathbb{{P}}}[(Y(\tilde{F})^{\mathcal{S}})(X_{t_1})]=\mathbb{E}_{\mathbb{{P}}}[(Y(\tilde{F})^{\mathcal{S}})(X_{0})]+
	\]
	\[
	\mathbb{E}_{\mathbb{P}}\left[\int_{0}^{t_1}L^{s,\mathcal{S}}(Y(\tilde{F}))(X_s)ds\right] + \mathbb{E}_{\mathbb{{P}}}\left[\int_{0}^{t_1} \partial_i^{s,\mathcal{S}} Y(\tilde{F})(X_s)\sigma^i_\alpha dW^\alpha_s\right].
	\]
 Since stochastic integrals are martingales we can write
 \[
 \mathbb{E}_{\mathbb{{P}}}\left[\left(\int_{0}^{t_1}H_\alpha(X_s)dW^\alpha_s\right)\left(\int_{0}^{t}L^{s,\mathcal{S}}(\tilde{F})(X_s)ds\right)\right]= \mathbb{E}_{\mathbb{{P}}}\left[\left(\int_{0}^{t}H_\alpha(X_s)dW^\alpha_s\right)\left(\int_{0}^{t}L^{s,\mathcal{S}}(\tilde{F})(X_s)ds\right)\right],
 \]
 and, by integration by parts,
 \[
 \mathbb{E}_{\mathbb{{P}}}\left[\left(\int_{0}^{t}H_\alpha(X_s)dW^\alpha_s\right)\left(\int_{0}^{t}L^{s,\mathcal{S}}(\tilde{F})(X_s)ds\right)\right]=\mathbb{E}\left[ \int_0^t \left\{ \left( \int_0^s H_\alpha (X_{\tau}) dW^\alpha_{\tau} \right) L^{s,\mathcal{S}}(\tilde{F})(X_s)ds \right\}\right. + \]
 \[ \left.+\int_0^t \left(\int_{0}^{s} L^{s,\mathcal{S}}(\tilde{F}) (X_\tau)d\tau\right) H_\alpha (X_{s}) dW^\alpha_{s}  +\left [\int_{0}^{t} H_\alpha (X_{s}) dW^\alpha_{s}, \int_{0}^{t}  L^{s,\mathcal{S}}(\tilde{F})(X_s)ds \right ] \right]  \]
 where the last term is a zero quadratic variation. Again by martingale property of the stochastic integral
\[ \mathbb{E}_{\mathbb{P}} \left[ \int_0^{t_1} \left( \int_0^s H_\alpha (X_{\tau})
	dW^\alpha_{\tau} \right) L^{s,\mathcal{S}}(\tilde{F})(X_s)ds \right] = \]
	\begin{equation} \label{scomposition2}
	=\mathbb{E}_{\mathbb{P}} \left[ \int_0^t \left\{ \left( \int_0^s H_\alpha
	(X_{\tau}) dW^\alpha_{\tau} \right) L^{s,\mathcal{S}}(\tilde{F})(X_s)ds + \left( \int_0^s H_\alpha
	(X_{\tau}) dW^\alpha_{\tau} \right)\partial^{s,\mathcal{S}}_i (\tilde{F}\circ \Phi_{\lambda})\sigma_\alpha^i dW^\alpha
	\right\} \right]
	\end{equation}
	Applying stochastic integration by parts formula to the second term in \eqref{scomposition2} we have
	\[ \int_0^t  \left( \int_0^s H_\alpha (X_{\tau}) dW^\alpha_{\tau}\right)
	\partial^{s,\mathcal{S}}_i (\tilde{F}\circ \Phi_{\lambda})\sigma_\alpha^i dW^\alpha = \left(\int_{0}^{t} H_\alpha (X_{s}) dW^\alpha_{s}\right) \left(\int_{0}^{t} \partial^{s,\mathcal{S}}_i (\tilde{F}\circ \Phi_{\lambda})\sigma_\alpha^i dW^\alpha\right) \]
	\[ -\int_0^t (\int_{0}^{s}\partial^{s,\mathcal{S}}_i (\tilde{F}\circ \Phi_{\lambda})\sigma_\alpha^i dW^\alpha) H_\alpha (X_{\tau}) dW^\alpha_{\tau}  -\left [\int_{0}^{t} H_\alpha (X_{s}) dW^\alpha_{s}, \int_{0}^{t}  \partial^{s,\mathcal{S}}_i (\tilde{F}\circ \Phi_{\lambda})\sigma_\alpha^i dW^\alpha  \right ],  \]
	with the quadratic variation equal to
	\[  \int_0^t (\partial^{s,\mathcal{S}}_i (\tilde{F}\circ \Phi_{\lambda})\sigma_\alpha^i) \cdot H_\alpha (X_s) ds. \]
	Summing the two stochastic expressions  we get
	\[ \left( \int_0^t L^{s,\mathcal{S}}(\tilde{F})(X_s)ds+ \int_0^t  \partial^{s,\mathcal{S}}_i (\tilde{F}\circ \Phi_{\lambda})\sigma_\alpha^i dW^\alpha  \right) \left( \int_0^t H_\alpha (X_{\tau}) dW^\alpha_{\tau}
	\right)  \]
	\[ - \int_0^t \left( \int_0^s L^{\tau,\mathcal{S}}(\tilde{F})(X_\tau)d\tau + \int_0^s  \partial^{\tau,\mathcal{S}}_i (\tilde{F}\circ \Phi_{\lambda})\sigma_\alpha^i dW^\alpha   \right) H_\alpha (X_{s}) dW^\alpha_{s} + \]
	\[ - \int_0^t ( \partial^{s,\mathcal{S}}_i (\tilde{F}\circ \Phi_{\lambda})\sigma_\alpha^i) \cdot H_\alpha (X_s) ds \]
	and inserting our final expressions in \eqref{derivativeslambda}, we finally get our integration by parts formula:
	\[
	\sum_{i=1}^{k}\mathbb{E}_{\mathbb{{P}}}[-m(t_i)L^{t_i,\mathcal{S}}(\tilde{F})+m(t_{i+1})L^{t_{i+1},\mathcal{S}}(\tilde{F})]+\mathbb{E}_{\mathbb{{P}}} \left[
	\tilde{F} (X_t) \left( \int_0^t H_\alpha (X_s) dW^\alpha_s \right) \right]
	+\mathbb{E}_{\mathbb{{P}}}[(Y(\tilde{F})^{\mathcal{S}})(X_t)]
	\]
	\[
	 - \mathbb{E}_{\mathbb{{P}}}[(Y(\tilde{F})^{\mathcal{S}})(X_0)]  =0
	\]
\end{proof}

\section*{Acknowledgement}

This work was partially supported by INdAM (Istituto Nazionale di Alta Matematica, Gruppo Nazionale per l’Analisi Matematica, la Probabilità e le loro Applicazioni and Gruppo Nazionale per la Fisica Matematica), Italy. The first Author was partially funded by the DFG under Germany’s Excellence Strategy—GZ 2047/1, Project-Id 390685813.The third Author was partially funded by DAAD-Reasearch Stay for University Academics and Scientists 2022 (Project ID-57588362).

\bibliographystyle{plain}
\bibliography{doob5}

\end{document}